\documentclass[12pt]{article}

\usepackage{amsmath,amssymb,amsthm}
\usepackage{tikz,framed}
\usetikzlibrary{calc}

\allowdisplaybreaks

\renewcommand{\geq}{\geqslant}
\renewcommand{\leq}{\leqslant}
\renewcommand{\ge}{\geqslant}
\renewcommand{\le}{\leqslant}

\usepackage{graphicx,wrapfig,caption,subcaption,color}
\usepackage[colorlinks=true,citecolor=black,linkcolor=black,urlcolor=blue]{hyperref}
\usepackage{enumerate}

\def\nfrac#1#2{{\textstyle\frac{#1}{#2}}}
\def\dfrac#1#2{\lower0.15ex\hbox{\large$\frac{#1}{#2}$}}

\def\nfrac#1#2{{\textstyle\frac{#1}{#2}}}
\def\dfrac#1#2{\lower0.15ex\hbox{\large$\frac{#1}{#2}$}}
\def\xvec{\boldsymbol{x}}
\def\Ivec{\boldsymbol{I}}

\newtheorem{theorem}{Theorem}[section]
\newtheorem{lemma}[theorem]{Lemma}

\newtheorem{corollary}[theorem]{Corollary}

\DeclareMathOperator{\tr}{tr}

\newcommand{\R}{\mathbb R} % R for reals
\newcommand{\Z}{\mathbb Z} % Z for integers
\newcommand{\N}{\mathbb N} % N for naturals
\newcommand{\E}{\mathbb E} % expected value
 
\newcommand{\cP}{\mathcal{P}}
\newcommand{\cG}{\mathcal{G}}
\newcommand{\cS}{\mathcal{S}}
\newcommand{\cD}{\mathcal{D}}
\newcommand{\cT}{\mathcal{T}}
\newcommand{\cI}{\mathcal{I}}
\newcommand{\cN}{\mathcal{N}}

\newcommand{\abs}[1]{\left| #1 \right|} % absolute values
\newcommand{\comments}[1]{} % a command so I can comment out big blocks of stuff

\newcommand{\qk}[1]{q_k(#1)}
\renewcommand{\Pr}{\mathbb{P}}  % CPC journal style

\title{Spanning trees in random regular uniform hypergraphs\thanks{Supported by the Australian Research Council Discovery Project DP190100977. This version includes a technical appendix (Appendix A) which is omitted in the journal version.} }

\author{
	Catherine Greenhill\\
	\small School of Mathematics and Statistics\\[-0.8ex]
	\small UNSW Sydney\\[-0.8ex]
	\small NSW 2052, Australia\\[-0.3ex]
	\small \texttt{c.greenhill@unsw.edu.au}\\
	\and
	Mikhail Isaev\thanks{Supported by the Australian Research Council Discovery Early Career Researcher Award DE200101045.}\\
	\small School of Mathematical Sciences \\[-0.8ex]
	\small Monash University\\[-0.8ex]
	\small VIC 3800, Australia\\
	\small\texttt{mikhail.isaev@monash.edu}
	\and
	Gary Liang\\
	\small School of Mathematics and Statistics\\[-0.8ex]
	\small UNSW Sydney\\[-0.8ex]
	\small Sydney NSW 2052, Australia\\
	\small \texttt{me@garyliang.net}
}

\date{26 March 2021}  

\begin{document}
	
	\maketitle
	
	\begin{abstract} 
Let $\mathcal{G}_{n,r,s}$ denote a uniformly random $r$-regular $s$-uniform
hypergraph on the vertex set $\{1,2,\ldots, n\}$. 
We establish a threshold result for the existence of a spanning tree
in $\mathcal{G}_{n,r,s}$, restricting to $n$ satisfying the necessary divisibility
conditions.  Specifically, we show that when $s\geq 5$, there is a positive
constant $\rho(s)$ such that for any $r\geq 2$, the probability that
$\mathcal{G}_{n,r,s}$ contains a spanning tree tends to~1 if $r > \rho(s)$, 
and otherwise this probability tends to zero.  
The threshold value $\rho(s)$ grows exponentially with $s$.
As $\mathcal{G}_{n,r,s}$ is connected with probability
which tends to~1, this implies that when $r \leq \rho(s)$, most $r$-regular $s$-uniform
hypergraphs are connected but have no spanning tree.
When $s=3,4$ we prove
that $\mathcal{G}_{n,r,s}$ contains a spanning tree with probability which
tends to~1, for any $r\geq 2$. Our proof also provides the 
asymptotic distribution of the number of spanning trees 
in $\mathcal{G}_{n,r,s}$ for all fixed integers $r,s\geq 2$.
 %such that $\Pr(\mathcal{G}_{n,r,s} \, \text{has a spanning tree})\to 1$. 
%tends to one as $n$ grows.  contains a spanning tree.
%We also calculate the asymptotic distribution of $Y_{\cG}$
	Previously, this asymptotic distribution was only
	known in the trivial case of 2-regular graphs,
	or for cubic graphs.
	\end{abstract}
	
	\section{Introduction}\label{s:intro}

A \textit{hypergraph} $H = (V,E)$ consists of a set of vertices $V$ and a multiset $E$ 
of non-empty multisubsets of $V$, which we call \textit{edges}.  A hypergraph is 
\emph{simple} if it has no repeated edges and no edge contains a repeated vertex.   
We focus on \emph{uniform hypergraphs}, where every edge has the same size, and say 
that a hypergraph is $s$-uniform if every edge has size $s$.  A graph is a simple 
2-uniform hypergraph.
	For more background on hypergraphs, see \cite{hypergraphs}.
	
%%% Misha's beautiful new introduction

In graph theory, a tree is a simple connected graph with no cycles, or equivalently a 
graph with the smallest number of edges among all connected graphs on a given vertex 
set.  There are several different ways to generalise this notion to hypergraphs, 
involving different definitions of acyclicity in hypergraphs~\cite{Baron}.
We use Berge acyclicity~\cite{Berge} to define trees in hypergraphs (hypertrees), see Section~\ref{s:preliminaries}.  
In particular, the definition implies that any two distinct edges in the tree intersect in at most one vertex, and hence an $s$-uniform tree with $n$ vertices has exactly 
$\frac{n-1}{s-1}$ edges. 
In fact, for $n$ satisfying this divisibility condition, a tree is a connected hypergraph on $n$ vertices with the smallest number of edges, exactly as in the graph case. 
We note also that our definition of trees in hypergraphs matches the definition given by
Boonyasombat in~\cite{Bo84}, while Siu refers to the trees we consider as
``traditional hypertrees''~\cite[Section~1.2.1]{Siu02}.

A spanning tree in a hypergraph $H$ is a spanning subhypergraph of $H$ which is a tree.
Just as trees in graphs are well-studied and extremely useful objects,
trees in hypergraphs have various applications in a wide variety of areas, including 
game theory~\cite{games}, relational databases~\cite{database},  
molecular optimisation~\cite{molecule} and network reliability~\cite{reach}. 
For example, Warme~\cite{warme} showed that the Steiner tree problem reduces to finding the minimum spanning tree in a hypergraph. 

It is well known that a graph contains a spanning tree if and only if it is connected. 
However, this relation does not extend to hypergraphs: that is, there exist connected 
hypergraphs without spanning trees. In fact, our results imply that asymptotically 
almost all $s$-uniform regular hypergraphs are like this, provided the degree is not 
too large (depending on $s$); see Theorem 1.1 and Lemma 1.4. 
The property of containing a spanning tree can be thought of as a kind of ``optimal 
connectedness'' of the hypergraph. This property is stronger than the usual notion
of connectedness,  which can be achieved using substantially overlapping edges and 
thus causing undesirable redundancy in various applications. 

%%%%%

	Given $r,s\geq 2$, let $n$ be a positive integer such that $s \mid rn$, 
 and let $\Gamma_{n,r,s}$ be the set of $r$-regular $s$-uniform simple hypergraphs on $[n]$. 
	Denote by $\mathcal{G}_{n,r,s}$ a hypergraph chosen uniformly at random from 
    $\Gamma_{n,r,s}$.
	Unless otherwise specified, all asymptotics in this paper are as $n\to\infty$, restricted to values of $n$ which
satisfy the necessary divisibility conditions: that is, restricted to the set
\[ 
\cN_{(r,s)} = \{ n\in\mathbb{Z}^+\, : \, s\mid rn \, \text{ and  }\, s-1\mid n-1\}.
\]
	Our main result establishes a degree threshold for the existence of a spanning tree in 
$\cG_{n,r,s}$, when $s\geq 5$,  and proves that $\mathcal{G}_{n,r,s}$ contains a spanning tree with 
probability which tends to 1 when $s\in \{ 2,3,4\}$, except for the case $(r,s)=(2,2)$.
A $2$-regular graph has a spanning tree if and only if it is connected (that is, forms a Hamilton cycle). Thus it follows from~\cite[Equation (11)]{wormald1999models} that as $n\to\infty$,
		\begin{equation} \label{eq:ham}
		\Pr(\cG_{n,2,2} \text{ contains a spanning tree}) \sim \frac{1}{2} e^{3/4}\, 
   \sqrt{\frac{\pi}{n}} \rightarrow 0.
		\end{equation}

\bigskip

	\begin{theorem} \label{thm:threshold}
	Let $s \ge 2$ be a fixed integer. If $s\geq 5$ then there exists a positive constant $\rho(s)$ such that for any fixed integer $r \ge 2$, as $n \to \infty$ along $\cN_{(r,s)}$,
		\[
		\Pr\left( \cG_{n,r,s} \text{ contains a spanning tree} \right) \longrightarrow 
 \begin{cases}
		1 & \text{if } r > \rho(s), \\
		0 & \text{if } r \le \rho(s).
		\end{cases} 
		\]
Specifically, $\rho = \rho(s)$ is the unique real number in $(2,\infty)$ such that 
		\begin{equation}
\label{rho-def}
 (s-1)^{\rho} (\rho - 1)^{s(\rho-1)} = \rho^{\rho s - \rho - s}(\rho s -\rho - s)^{(\rho s -\rho - s)/(s-1)}.
		\end{equation}
For $s\geq 5$ we have $\rho^{-}(s) < \rho(s) < \rho^+(s)$ where
		\[
	\rho^{-}(s) = \frac{e^{s-2}}{s-1} - \frac{s-1}{2},\qquad
    \rho^{+}(s) = \frac{e^{s-2}}{s-1} - \frac{s-3}{2}.
		\]
Furthermore, as $s\to\infty$,
\[ \rho(s) = \frac{e^{s-2}}{s-1}- \frac{s^2-3s+1}{2(s-1)} + O_s(s^4\, e^{-s}).\]
Finally, if $s\in \{2,3,4\}$ then for any fixed integer $r\geq 2$,
\[ \Pr\left( \cG_{n,r,s} \text{ contains a spanning tree} \right) \longrightarrow 
 \begin{cases}
    1 & \text{if $(r,s)\neq (2,2)$,}\\
    0 & \text{if $(r,s)=(2,2)$.}\end{cases}
\] 
	\end{theorem}

The value of $\rho(s)$ for $s=5,\ldots, 12$ is displayed in Table \ref{tab:threshold}, together with 
the bounds $\rho^{-}(s)$ and $\rho^+(s)$.  All values are rounded to 3 decimal places. 
	
	\renewcommand{\arraystretch}{1.2}%
		\begin{table}[ht!]
			\centering
			\begin{tabular}{|c||c|c|c|c|c|c|c|c|}
				\hline
				$s$ & $5$ & $6$ & $7$ & $8$ & $9$ & $10$ & $11$ & $12$ \\
				\hline 
				$\rho^{-}(s)$ & $3.021$ & $8.420$ & $21.736$ & $54.133$ & $133.079$ & $326.718$ & $805.308$ & $1996.906$  \\ \hline
				$\rho(s)$ & $3.029$ & $8.706$ & $22.142$ & $54.606$ & $133.588$ & $327.245$ & $805.844$ & $1997.444$  \\ \hline
				$\rho^{+}(s)$ & $4.021$ & $9.420$ & $22.736$ & $55.133$ & $134.079$ & $327.718$ & $806.308$ & $1997.906$ \\ \hline
			\end{tabular}
			\caption{Values of $\rho(s)$ for $s = 5, \ldots, 12$, together with our bounds.}
			\label{tab:threshold}
		\end{table}
	
Let $Y_{\cG}$ be the number of spanning trees in $\cG_{n,r,s}$. 
This random variable is our main object of study.
Using asymptotic enumeration methods, Aldosari and Greenhill~\cite[Corollary 1.2.]{AG2} 
established the following asymptotic expression for $\E Y_{\cG}$ when $s\geq 3$: 
	\begin{equation} \label{eq:EY_G}
	\E Y_{\cG} \sim \exp\left( \frac{rs-r-1}{2(r-1)}\right) \frac{(s-1)\,\sqrt{r-1} }{n \, (rs-r-s)^{\frac{s+1}{2(s-1)}}} \left( \frac{(s-1)^{r}\, (r-1)^{(r-1)s}}{r^{rs-r-s}\, (rs-r-s)^{\frac{rs-r-s}{s-1}}} \right)^{n/s}.
	\end{equation}
(In fact a more general result is proved in~\cite{AG2}, which covers irregular degree
sequences and allows $s$ and the maximum degree to grow slowly with $n$.)
In the graph case, the asymptotic formula for $\E Y_{\cG}$ was known up to a constant factor
by the results of McKay~\cite{mckay1981} (who also considered irregular, slowly-growing degrees), 
and then this constant factor was calculated 
precisely in~\cite[Theorem~1.1]{spanning}.

\medskip

The argument used to prove Theorem~\ref{thm:threshold} also provides 
the asymptotic distribution of $Y_{\cG}$, for any parameters $r,s\geq 2$.
%% such that $\Pr(Y_{\cG}>0)\to 1$.  
%% If $\Pr(Y_{\cG}>0)\to 0$ then the asymptotic distribution of $Y_{\cG}$ is 
%% trivial (a point mass at zero).
%The statement of Theorem~\ref{thm:threshold}
%can be used to decide whether Theorem~\ref{thm:distribution} applies.

	\begin{theorem} \label{thm:distribution}
Let $r,s,\geq 2$ be fixed integers.
For all positive integers $j$, define
		\[
		\lambda_j = \frac{(r-1)^j (s-1)^j}{2j} \quad \text{and} \quad \zeta_j = \frac{\left(\frac{r}{r-1} - s + 1 \right)^j - 2}{(r-1)^j (s-1)^j}.
		\]
Let $J(2) = 3$ and $J(s) =2$ for $s \ge 3$. 
If $\Pr(Y_{\cG}>0) \rightarrow 1$ as $n\to\infty$ along $\cN_{(r,s)}$,
then the asymptotic distribution of $Y_{\cG}$ satisfies
\[
		\frac{Y_{\cG}}{\E Y_{\cG}} \overset{d}{\longrightarrow} \prod_{j=J(s)}^\infty (1+ \zeta_j)^{Z_j} \, e^{-\lambda_j \zeta_j}
\]
		where 
$Z_j = \operatorname{Po}(\lambda_j)$ are independent Poisson random variables.
Otherwise, the asymptotic distribution of $Y_{\cG}$ is a point mass at zero.
	\end{theorem}

%Greenhill, Kwan and Wind~\cite{spanning} determined the asymptotic distribution of 
%spanning trees for random cubic graphs, that is, when $(r,s) = (3,2)$.  Our proof
%provides a generalisation of this result to random $r$-regular $s$-uniform hypergraphs 
%for any fixed $r,s\geq 2$ with $(r,s)\neq (2,2)$. 
	
%The only previously-known case of 
Previously, the result of Theorem~\ref{thm:distribution} was only known
for two values of $(r,s)$:  when $(r,s)=(2,2)$ the result follows trivially
from (\ref{eq:ham}),  while for $(r,s)=(3,2)$ (cubic graphs), the 
asymptotic distribution of $Y_{\cG}$ was obtained by
Greenhill, Kwan and Wind~\cite[Theorem~1.2]{spanning}.
The authors of~\cite{spanning} also conjectured an expression for the asymptotic
distribution of $Y_{\cG}$ when $s=2$ and $r\geq 4$.  Substituting $s=2$ into
Theorem~\ref{thm:distribution}  verifies that their conjecture is true.

\begin{corollary}\label{conjecture-holds}
The conjecture given in~\emph{\cite[Conjecture 1.3]{spanning}} holds.
That is, the number $Y_{\mathcal{G}}$ of spanning trees in a random $r$-regular
graph satisfies 
\[
\frac{Y_{\cG}}{\E Y_{\cG}} \overset{d}{\longrightarrow} 
  \prod_{j=3}^\infty (1+ \zeta_j)^{Z_j} \, e^{-\lambda_j \zeta_j}
\]
with
\[ \lambda_j = \frac{(r-1)^j}{2j},\qquad \zeta_j = -\frac{2(r-1)^j-1}{(r-1)^{2j}}\]
for $j\geq 3$.
\end{corollary}

\medskip

Theorems~\ref{thm:threshold} and~\ref{thm:distribution} are proved using the small subgraph conditioning method~\cite{robinson1992almost}.
While many structural results about random regular graphs have been proved using 
this method (see~\cite{janson1995random,wormald1999models} for surveys),
there are only two previously-known results for hypergraphs with $s\geq 3$. 
Cooper, Frieze, Molloy and Reed~\cite{cooper1996perfect} gave a threshold result for the existence of a perfect 
matching in $\mathcal{G}_{n,r,s}$, while
	Altman, Greenhill, Isaev and Ramadurai~\cite{loose} proved a threshold result for the existence of
loose Hamilton cycles in $\cG_{n,r,s}$.
Keeping only the most significant term, the threshold values for spanning trees,
loose Hamilton cycles and perfect matchings are approximately
\[ \frac{e^{s-2}}{s-1}\,\,\, \text{(spanning trees)},\qquad \frac{e^{s-1}}{s-1} \,\,\,
  \text{(loose Hamilton cycles)},\qquad e^{s-1} \,\,\,\text{(perfect matchings)}\]
respectively.   Hence (restricting to values of $n$ satisfying the relevant
divisibility conditions in each case), for a fixed $s$, as $r$ increases,
spanning trees appear first, followed by loose Hamilton cycles and then perfect matchings.

We close this section with some comments on connectedness in random
regular uniform hypergraphs. 
It is well known that with probability tending to~1,
random $r$-regular graphs are connected (indeed, $r$-connected)
whenever $r\geq 3$, see~\cite{bollobas-1981,wormald-connectivity}. 
Dumitriu and Zhu~\cite{DZ} recently used spectral methods to investigate
expansion properties of $\mathcal{G}_{n,r,s}$, and hence inferred~\cite[Lemma 6.2]{DZ}
that $\Pr(\mathcal{G}_{n,r,s} \text{ is connected})\to 1$ when $r\geq s\geq 3$.  
For completeness we sketch a more elementary argument which covers
all $(r,s)\neq (2,2)$.

\begin{lemma}
\label{lem:connectivity}
Let $r,s\geq 2$ be fixed integers with $(r,s)\neq (2,2)$.  Then 
\[ \Pr(\cG_{n,r,s} \text{ is connected}) \to 1\]
as $n$ tends to infinity along $\mathcal{N}_{(r,s)}$.
\end{lemma}

\begin{proof}
Suppose that $A\subset [n]$ is a subset of vertices with $|A|=a$, where
$1\leq a\leq n/2$.
The probability that $(A,[n]-A)$ is a cut in $\cG_{n,r,s}$ is
\begin{equation}
\label{eq:prob}
    \frac{|\Gamma_{a,r,s}|\, |\Gamma_{n-a,r,s}|}{|\Gamma_{n,r,s}|}.
 \end{equation}
It follows from (\ref{Gammanrs}), (\ref{eq:simple-graph}) and (\ref{eq:simple-hypergraph}) that
\[ |\Gamma_{n,r,s}|  = \Theta(1)\, \,
 \frac{(rn)!}{(rn/s)!\, (s!)^{rn/s}\, (r!)^n}.\] 
Substituting this into (\ref{eq:prob}) and summing over all $A$ with $1\leq |A|\leq n/2$, we conclude that
the probability that $\cG_{n,r,s}$ is disconnected is $o(1)$ whenever 
$r,s\geq 2$ and $(r,s)\neq (2,2)$.
\end{proof} 
	
	\section{Preliminaries}\label{s:preliminaries}

Throughout, $\mathbb{N}$ denotes the nonnegative integers and $(a)_b = a(a-1)\cdots (a-b+1)$
is the falling factorial.
	
A \textit{$1$-cycle} (or \textit{loop}) is a hypergraph consisting of one edge which contains a repeated vertex. 
A \textit{$2$-cycle} is a hypergraph consisting of two edges which intersect in at least $2$ vertices. 
	For $j \ge 3$, a \textit{$j$-cycle} is a hypergraph with $j$ edges which can be labelled $e_1, e_2, \dots, e_{j}$ such that there exists distinct vertices $v_1, \dots, v_j$ where $v_i \in e_i \cap e_{i+1}$ for $i = 1, \dots, j$ (where $e_{j+1} \equiv e_1$).  
	
A \emph{(Berge) path} in $H$ consists of a sequence $v_0,e_1,v_1,e_2,\ldots, e_j,v_j$ where $v_0,v_1,\ldots, v_j$ are distinct vertices, $e_1,\ldots, e_j$
are distinct edges, and $v_{i-1}, v_i\in e_i$ for all $i=1,\ldots, j$.
A hypergraph is \emph{connected} if there is a path between every pair of
vertices.
A \emph{spanning tree} $T$ in
a hypergraph $H$ is a connected spanning subhypergraph of $H$ which contains no 
$j$-cycles for all positive integers $j$.
In particular, as $T$ contains no 2-cycles it follows that edges of $T$ 
overlap in at most one vertex.
	
An $s$-uniform tree with $t$ edges has $(s-1)t+1$ vertices, and the 
	number of (labelled) $s$-uniform trees on $n = (s-1)t + 1$ vertices is
	given by
	\begin{equation}
	\label{s-uniform-spanning-trees}
	\frac{n^{t-1}\, (n-1)!}{t! \, ((s-1)!)^t}.
	\end{equation}
	This formula was proved in~\cite[Corollary 1]{lavault} 
	and~\cite[Theorem 2]{sivasubramanian2006spanning}.
	See also~\cite{shannigrahi}. 
	Note that in the graph case $s=2$ we recover Cayley's formula. 
	
Suppose that $\boldsymbol{\delta} = (\delta_1, \dots, \delta_n)$
is the degree sequence of an $s$-uniform tree on $n$ vertices.
Then $\boldsymbol{\delta}$ is a sequence of $n$ positive integers such that
	\[
	\delta_1 + \dots + \delta_n = \frac{s(n-1)}{s-1} = st,
	\]
	and any such sequence $\boldsymbol{\delta}$ is called a \emph{tree degree sequence}. 
	Bacher~\cite[Theorem 1.1]{bacher} proved that the number of
	labelled $s$-uniform trees on $n=(s-1)t+1$ vertices with
	degree sequence $\boldsymbol{\delta}$ is
	\begin{equation} \label{degree-sequence-hypertree}
	\binom{t - 1}{\delta_1 -1, \dots, \delta_n - 1}\, \frac{(n-1)!}{t! \, ((s-1)!)^t} = \frac{(s-1)(n-2)! }{((s-1)!)^{\frac{n-1}{s-1}}} \prod_{i=1}^n \frac{1}{(\delta_i-1)!}.
	\end{equation}
	This generalises the formula for the graph case proved by Moon~\cite{moon1970counting}.

Now we introduce some special families of cycles which will be used in our
analysis.
	An $s$-uniform $1$-cycle is \textit{loose} if the edge contains $s-1$ distinct vertices, and an $s$-uniform $2$-cycle is \textit{loose} if the intersection of the two edges has size $2$. For $j \ge 3$, an $s$-uniform $j$-cycle is \textit{loose} if
	\[
	\abs{e_k \cap e_\ell} = \left\{ \begin{array}{ll}
	1 & \text{if } k-\ell \equiv \pm 1 \pmod{j}, \\
	0 & \text{otherwise}
	\end{array} \right.
	\]
	for $k \ne \ell$. A loose $j$-cycle $C$ contains $(s-1)j$ vertices.
	
	Let $v$ be a vertex in a loose $j$-cycle $C$. We say $v$ is \textit{$C$-external} (or \textit{external}) if $v$ has degree $2$ in $C$. 
	Otherwise, we say $v$ has degree $1$ and we say it is \textit{$C$-internal} (or \textit{internal}).

	\subsection{Configuration model}\label{s:config}
	
	The configuration model for regular uniform
	hypergraphs is a generalisation of the configuration
	model for graphs, introduced by by Bollob{\'a}s~\cite{bollobas1980probabilistic}.
	Take $rn$ \emph{points} in $n$ cells, $B_1,\ldots, B_n$,
	each containing $r$ points.  Then partition the points 
	into $rn/s$ subsets of size $s$, called \textit{parts}.  Each such partition $P$ corresponds
	to a hypergraph $G(P)$, which may not be simple, obtained by replacing
	each cell by a vertex and replacing each part $\{ a_1,\ldots, a_s\}$
	in the partition by an edge $\{ v_{i_1},\ldots, v_{i_s}\}$ such that
	$a_j\in B_{i_j}$ for $j=1,\ldots, s$.
	If a part contains more than one point from the same cell then the corresponding
	edge in $P$ is a loop, and if the partition has two parts which contain
	points from precisely the same cells, with the same multiplicity,
	then these parts produce a repeated edge in $G(P)$.
	
	Each simple hypergraph corresponds to precisely $(r!)^n$ partitions, so an $r$-regular $s$-uniform simple hypergraph can be chosen uniformly at random by choosing a partition uniformly at random and rejecting the result if it has loops or multiple edges.
	A partition $P$ is said to be simple if $G(P)$ is a simple hypergraph.
	
	Denote the set of possible partitions by $\Omega_{n,r,s}$, and let 
	$\mathcal{P}_{n,r,s}$ be a partition chosen uniformly at random from $\Omega_{n,r,s}$.
	
	A \textit{subpartition} $P'$ is a subset of a partition $P \in \Omega_{n,r,s}$. 
	A subpartition of $P$ projects to a subhypergraph of $G(P)$.
	
	The configuration model allows us to prove some properties of $\cG_{n,r,s}$ by performing computations in $\cP_{n,r,s}$. When $s\mid t$, the number of partitions of a set of $t$ points into $t/s$ parts of size $s$ is
	\[
	p(t) = \frac{t!}{(t/s)! (s!)^{t/s}}.
	\]
	Hence the number of partitions in $\Omega_{n,r,s}$ is
	\begin{equation} \label{eq:partitions-hypergraph}
	\abs{\Omega_{n,r,s}} = p(rn) = \frac{(rn)!}{(rn/s)! \, (s!)^{rn/s}}.
	\end{equation}
	Therefore, the number of $r$-regular $s$-uniform simple hypergraphs on $n$ vertices is precisely
%% only need this equation labelled if we include the connectivity proof, Appendix B
	\begin{equation}
\label{Gammanrs}
	\abs{\Gamma_{n,r,s}} =\frac{(rn)! \, \Pr(\text{Simple})}{(rn/s)! \, (s!)^{rn/s} (r!)^n},
	\end{equation}
	where ``Simple'' is the event that the partition is simple.
	Thus, an asymptotic formula for $\abs{\Gamma_{n,r,s}}$ can be found by estimating $\Pr(\text{Simple})$. 
	%%%
	When $s=2$ the event ``Simple''
	means no 1-cycles or 2-cycles, and Bender and Canfield~\cite{bender1978asymptotic}
	showed that
	\begin{equation} \label{eq:simple-graph}
	\Pr(\text{Simple}) \sim \exp\left( -(r^2-1)/4 \right).
	\end{equation}
	Cooper, Frieze, Molloy and Reed showed in \cite{cooper1996perfect} that for fixed integers $r \ge 2$ and $s \ge 3$, 
	\begin{equation} \label{eq:simple-hypergraph}
	\Pr(\text{Simple}) \sim \exp\left( -(r-1)(s-1)/2 \right).
	\end{equation}

	\subsection{Small subgraph conditioning method for hypergraphs}\label{s:SSCM}
	
	Robinson and Wormald showed in~\cite{robinson1992almost,robinson1994almost} that 
	almost all $r$-regular graphs are Hamiltonian, for any fixed $r \ge 3$,
using an analysis of variance technique now known as the \textit{small subgraph conditioning method}. %Janson \cite{janson1995random} stated the small subgraph conditioning method in its current form in 1995. 
	We restate the small subgraph conditioning method from \cite{janson1995random} (with slightly different notation).
	
	\begin{theorem}[Janson \protect{\cite[Theorem 1]{janson1995random}}] \label{thm:subgraph}
		Let $\lambda_j > 0$ and $\zeta_j \ge -1$, $j = 1, 2, \dots$, be constants and suppose that for each $n$ there are random variables $X_{j,n}$, $j = 1,2, \dots$, and $Y_n$ (defined on the same probability space) such that $X_{j,n}$ is a nonnegative integer valued and $\E Y_n \ne 0$ (at least for large $n$), and furthermore the following conditions are satisfied:
		\begin{enumerate}
			\item[\emph{(A1)}] $X_{j,n} \overset{d}{\longrightarrow} Z_j$ as $n \to \infty$ jointly for all $j$, where $Z_j \sim \operatorname{Po}(\lambda_j)$ are independent Poisson random variables;
			\item[\emph{(A2)}] For any finite sequence $x_1, \dots, x_m$ of nonnegative integers,
			\[
			\frac{\E(Y_n | X_{1,n} = x_1, \dots, X_{m,n} = x_m)}{\E Y_n} \to \prod_{j=1}^m (1+ \zeta_j)^{x_j} e^{- \lambda_j \zeta_j} \quad \text{ as } n \to \infty;
			\]
			\item[\emph{(A3)}] $\displaystyle \sum_{j \ge 1} \lambda_j \zeta_j^2 < \infty$;
			\item[\emph{(A4)}] $\dfrac{\E Y_n^2}{(\E Y_n)^2} \to \exp\left( \displaystyle \sum_{j \ge 1} \lambda_j \zeta_j^2 \right) \quad \text{ as } n \to \infty$.
		\end{enumerate}
		Then
		\[
		\frac{Y_n}{\E Y_n} \,\,\, \overset{d}{\longrightarrow} \,\,\, W = \prod_{j=1}^\infty (1+\zeta_j)^{Z_j} e^{-\lambda_j \zeta_j} \quad \text{ as } n \to \infty;
		\]
		moreover, this and the convergence in \emph{(A1)} hold jointly. The infinite product defining $W$ converges asymptotically almost surely and in $L^2$, with 
		\[
		\E W = 1 \qquad \text{and} \qquad \E W^2 = \exp\left( \sum_{j \ge 1} \lambda_j \zeta_j^2 \right) = \lim_{n \to \infty} \frac{\E Y_n^2}{(\E Y_n)^2}.
		\] 
Furthermore, the event $W=0$ equals, up to a set of probability zero,
the event that $Z_j>0$ for some $j$ with $\zeta_j = -1$.
		In particular, $W > 0$ almost surely if and only if every $\zeta_j > -1$. 
	\end{theorem}

(In the above statement, we have corrected a typographical error from~\cite{janson1995random},
which had $W>0$ instead of $W=0$ in the second-last sentence.)
	
	Janson remarks in \cite{janson1995random} that the index set $\Z^+$ may be replaced by any other countably infinite set, and that $0^0$ is defined to be $1$. 
	%
	%\begin{theorem}[Bollob{\'a}s \cite{bollobas1980probabilistic}] \label{thm:poisson}
	%	For fixed $r$, let $X_{j,n}$ be the number of cycles of length $j$ in the random multigraph coming from a partition in $\mathcal{P}_{n,r}$. For $k \ge 1$, $X_{1,n}, \dots, X_{k,n}$ are asymptotically independent Poisson random variables with means $\lambda_j = \frac{(r-1)^j}{2j}$.
	%\end{theorem}
	We will apply Theorem \ref{thm:subgraph} with the following random variables:
	\begin{itemize}
		\item Let $Y$ be the number of $s$-uniform spanning trees in a random partition $P\in\cP_{n,r,s}$.
		\item Let $X_1$ be the number of $1$-cycles in a random partition $P\in\cP_{n,r,s}$.
		\item For $j \ge 2$, let $X_j$ be the number of loose $j$-cycles in a random partition $P\in\cP_{n,r,s}$.
	\end{itemize}
With $r,s\geq 2$ fixed, it is well-known that $X_j \to Z_j$ as $n \to \infty$, where $Z_j$ are asymptotically independent Poisson random variables with mean
	\begin{equation} \label{eq:lambda}
	\lambda_j = \frac{(r-1)^j(s-1)^j}{2j}.
	\end{equation}
This was proved for graphs ($s=2$) by Bollob{\' a}s~\cite{bollobas1980probabilistic}, 
and by Cooper, Frieze, Molloy and Reed~\cite{cooper1996perfect} when $s\geq 3$.
	To be more precise, Cooper, Frieze, Molloy and Reed worked with the random variable $X_j'$, the number of $j$-cycles (not necessarily loose), and showed that $X_j'$ has the same asymptotic distribution as $X_j$, as the contribution to $X_j'$ from non-loose $j$-cycles forms a negligible fraction of $X_j'$. This verifies that (A1) of Theorem \ref{thm:subgraph} holds.

	In order to verify condition (A2), the following lemma is helpful.
	\begin{lemma}[\protect{Janson \cite[Lemma 1]{janson1995random}}]
		Let $\lambda_j' \ge 0$, $j = 1, 2, \dots$ be constants. Suppose that \emph{(A1)} holds, that $Y_n \ge 0$ and that
		\[
		\text{\emph{(A2$'$)}} \qquad \frac{\E(Y_n (X_{1,n})_{x_1} \cdots (X_{m,n})_{x_m})}{\E Y_n}  \to \prod_{j=1}^m (\lambda_j' )^{x_j} \quad \text{ as } n \to \infty,
		\]
		for every finite sequence $x_1, \dots, x_m$ of nonnegative integers. Then condition \emph{(A2)} holds with $\lambda_j' = \lambda_j (1 + \zeta_j)$.
	\end{lemma}
	
	There are some challenges when applying the small subgraph conditioning method to 
	regular uniform hypergraphs with $s\geq 3$.  In the graph case, a partition is simple precisely when $X_1=X_2=0$.
	For hypergraphs with $s\geq 3$, this is no longer true: a hypergraph is simple if and only if it
	has no 1-cycles and no repeated edges: some 2-cycles are allowed, as long as the 
two edges overlap in between $2$ and $s-1$ vertices.
	
	Fortunately, we can translate some asymptotic properties from the configuration model to random hypergraphs. For any event $\mathcal{E} \subseteq \Omega_{n,r,s}$, we have
	\begin{equation}
	\Pr(P \in \mathcal{E} \, | \, \text{Simple} ) \le \frac{\Pr(P \in \mathcal{E})}{\Pr(\text{Simple})}. 
	\label{eq:conditional-simple}
	\end{equation}
	
	Altman, Greenhill, Isaev and Ramadurai proved the following lemma in~\cite{loose}.
	
	\begin{lemma}[\protect{\cite[Lemma 2.1]{loose}}] \label{lemma:translate}
		Fix integers $r,s \ge 2$. For any positive integer $n$ such that $s \mid rn$, let $\widehat{P}$ be a uniformly random partition in $\Omega_{n,r,s}$ with no $1$-cycles, let $P_{S}$ be a uniformly random simple partition in $\Omega_{n,r,s}$ and let $P$ be a uniformly random partition in $\Omega_{n,r,s}$. Let $Y: \Omega_{n,r,s} \to \Z$ be a random variable. Then as $n \to \infty$ along integers such that $s \mid rn$, the following two properties hold.
		\begin{enumerate}
			\item[\emph{(a)}] If $\Pr(Y(P) \in A) = o(1)$, then $\Pr(Y(P_{S}) \in A) = o(1)$ for any $A \subseteq \Z$.
			\item[\emph{(b)}] $\Pr(Y(\widehat{P}) \in A) - \Pr(Y(P_{S}) \in A) = o(1)$ for any $A \subseteq \Z$.
		\end{enumerate}
	\end{lemma}
	Property (a) follows from \eqref{eq:conditional-simple}, while property (b) follows from \eqref{eq:simple-hypergraph} and the fact that the probability that two parts in a random partition $P\in \cP_{n,r,s}$ give rise to a repeated edge is $o(1)$ (as remarked by Cooper, Frieze, Molloy and Reed in \cite{cooper1996perfect}). 
	
	Property (b) essentially tells us that the distribution of $Y$ that arises from conditioning on $X_1 = 0$ is asymptotically equivalent to the distribution $Y$ conditioned on ``Simple". This allows us to apply the following corollary, very slightly adapted from~\cite[Corollary~2.6]{loose}, which will be useful in the proof of Theorem \ref{thm:distribution}. 

	\begin{corollary} \label{cor:translate-back-to-hypergraphs}
		Suppose that $Y_n$ and $X_{j,n}$ satisfy conditions \emph{(A1)}--\emph{(A4)} of Theorem \ref{thm:subgraph}. Let $\widehat{Y}_n$ be the random variable obtained from $Y_n$ by conditioning on the event $X_{1,n} = 0$. Then
		\[
		\frac{\widehat{Y}_n}{\E Y_n} \,\,\, \overset{d}{\longrightarrow}  \,\,\,e^{-\lambda_1 \zeta_1} 
		\, \prod_{j=2}^\infty (1+\zeta_j)^{Z_j} e^{-\lambda_j \zeta_j} \quad \text{as} \quad n \to \infty.
		\]
		Moreover, if $\zeta_j > -1$ for all $j \ge 2$ then asymptotically almost surely $\widehat{Y}_n > 0$.
	\end{corollary}
	
\begin{proof}
The statement of~\cite[Corollary~2.6]{loose} made the assumption that $\zeta_j > -1$ 
for all $j\geq 1$, and the proof used the final statement of Theorem~\ref{thm:subgraph}.  However,
we can drop the assumption that $\zeta_1 > -1$
if we instead apply the second-last statement from Theorem~\ref{thm:subgraph}.
\end{proof}

\section{First moment}\label{s:first}
	
We fix integers $r,s \ge 2$, where $(r,s) \ne (2,2)$, and work in the configuration model $\cP_{n,r,s}$, where $s \mid rn$ and $n = (s-1)t +1$ for some $t \in \N$. 
	
	\begin{lemma} \label{lem:EY-hypergraph}
Let $r,s \ge 2$ be fixed integers with $(r,s) \ne (2,2)$.  
Then as $n\to\infty$ along $\cN_{(r,s)}$,
		\[ %\label{eq:EY-hypergraph}
	\E Y \sim \frac{(s-1)\, \sqrt{r-1} }{n \, (rs-r-s)^{\frac{s+1}{2(s-1)}}} 
\left( \frac{(s-1)^{r}(r-1)^{(r-1)s}}{r^{rs-r-s} (rs-r-s)^{\frac{rs-r-s}{s-1}}} \right)^{n/s}.
		\]
	\end{lemma}
	
	\begin{proof}
	Let $\cT_{n}$ be the set of all $s$-uniform trees on $n$ vertices. 
For a random partition $P\in \mathcal{P}_{n,r,s}$, we can write 
		\[
\E Y = \sum_{P_T:G(P_T) \in \cT_{n}} \Pr(P_T \subseteq P) = \sum_{P_T:G(P_T) \in \cT_{n}} \frac{ \abs{\{P \in \Omega_{n,r,s}: P_T \subseteq P \}} }{\abs{\Omega_{n,r,s}}}.
		\]
	Selecting $P_T$ uses up $st$ points so, for a given $P_T$, the size of $\{P \in \Omega_{n,r,s}: P_T \subseteq P \}$ is $p(rn-st)$. This is the number of ways to partition the remaining $rn-st$ points after the points of the tree are selected. 
	Recall that $\cD_n$ is the set of possible tree degree sequences on $n$ vertices. Given $\boldsymbol{\delta} \in \cD_n$, define $\cT_n(\boldsymbol{\delta})$ to be the set of trees with a degree sequence $\boldsymbol{\delta}$. 
We can write
	\begin{equation}
\label{exp-intermediate}
		\abs{\Omega_{n,r,s}} \E Y =  p(rn-st) \sum_{\boldsymbol{\delta} \in \cD_n}\sum_{T \in \cT_n(\boldsymbol{\delta})} \sum_{P_T:G(P_T) = T} 1.
		\end{equation}
Consider a subpartition that projects to a given tree $T \in \cT_{n}(\boldsymbol{\delta})$, for a given 
$\boldsymbol{\delta} \in \cD_n$. Exactly $\delta_i$ of the points in cell $i$ must contribute to $P_T$, 
and there are $(r)_{\delta_j}$ ways to choose and order these points. So there are 
$\prod_{j=1}^n (r)_{\delta_j}$ possible subpartitions $P_T$ which project to the given spanning tree $T$.
Therefore, using (\ref{degree-sequence-hypertree}) for the second line, the number of
subpartitions $P_T$ which project to some spanning tree is
		\begin{align}
\sum_{\boldsymbol{\delta} \in \cD_n} \abs{\cT_n(\boldsymbol{\delta})} \prod_{j=1}^n (r)_{\delta_j} 
 &= 
  \frac{(s-1)(n-2)!}{((s-1)!)^t} \left(\sum_{\boldsymbol{\delta} \in \cD_n} \prod_{j=1}^n \frac{(r)_{\delta_j}}{(\delta_j-1)!}\right) \nonumber \\
&= \frac{(s-1)(n-2)!}{((s-1)!)^t} \,  [z^{st}]\left(\sum_{i=1}^\infty \frac{(r)_{i}}{(i-1)!} \, z^i \right)^n \nonumber \\ 
   &= \frac{(s-1)(n-2)!}{((s-1)!)^t}\,  [z^{st}]\left(\sum_{i=1}^{\infty} rz \binom{r-1}{i-1} z^{i-1} \right)^n \nonumber \\
 &= \frac{r^n (s-1)(n-2)!}{((s-1)!)^t} \, \binom{(r-1)n}{t-1}.  \label{number-of-tree-subpartitions}
		\end{align}
Here square brackets denotes coefficient extraction.
Substituting (\ref{number-of-tree-subpartitions}) into (\ref{exp-intermediate}) and
applying \eqref{eq:partitions-hypergraph} gives, by definition of $t$,
\begin{align}
	\E Y 
	  &= \frac{p(rn-st)}{p(rn)} \cdot \frac{r^n(s-1)(n-2)!}{((s-1)!)^t}  \binom{(r-1)n}{t-1}\nonumber \\
		&=r^n s^{\frac{n-1}{s-1}}  \,
   \frac{((r-1)n)!\, (n-1)!\, (rn/s)!}{(rn)!\, \left(\frac{n-1}{s-1}\right)! \, 
     \left(\frac{(rs-r-s)n + s}{s(s-1)}\right)!}.
 \label{eq:EY}
		\end{align}
		The result follows by applying Stirling's approximation.
	\end{proof}

\bigskip
For future reference we note that by \eqref{eq:partitions-hypergraph}, Lemma~\ref{lem:EY-hypergraph} and Stirling's approximation,
	\begin{equation}
\label{for-future-reference}
	\abs{\Omega_{n,r,s}} \E Y \sim \frac{(s-1)\, \sqrt{(r-1)s\ } }{n(rs-r-s)^{\frac{s+1}{2(s-1)}}} 
 \left(\frac{r^s\, (r-1)^{(r-1)s}\, n^{r(s-1)}}{(rs-r-s)^{\frac{rs-r-s}{s-1}}\, ((s-2)!)^{r}\, 
  e^{r(s-1)}} \right)^{n/s}.
	\end{equation}
\bigskip
	
	In Section~\ref{s:threshold} we characterise pairs $(r,s)$ for which $\E Y$ tends to infinity.
	This provides the threshold function $\rho(s)$ when $s\geq 5$ and, using (\ref{eq:conditional-simple}),
	provides the negative half of the threshold result.  In order to complete the proof,
	we must apply small subgraph conditioning.

	\section{Effect of short cycles} \label{ch:EYX}

Fix a positive integer $m$ and a sequence $\xvec = (x_1,\ldots,x_m) \in \N^m$. 
Write $\ell = x_1 + \dots + x_m$.  
Let $\mathcal{S}(\xvec)$ be the set of sequences $(P_1,\ldots, P_\ell)$ of subpartitions 
such that $G(P_1),\ldots, G(P_{x_1})$ are distinct 1-cycles, and 
\[ G(P_{x_1 + \cdots + x_{j-1} + 1}),\ldots, G(P_{x_1 + \cdots + x_{j-1} + x_j})\]
are distinct loose $j$-cycles, for $j=2,\ldots, m$.
Then let $\mathcal{S}^\ast(\xvec)$ be the set of all $(P_1,\ldots, P_\ell)\in \mathcal{S}(\xvec)$ such that the cycles $G(P_1),\ldots, G(P_\ell)$ are vertex-disjoint. For a random partition $P \in \cP_{n,r,s}$, we can write
	\begin{equation}
\E[ Y (X_1)_{x_1}\cdots (X_m)_{x_m}] = \sum_{(P_1,\ldots, P_{\ell})\in\mathcal{S}(\xvec)}\,\,
	\sum_{P_T: G(P_T) \in \cT_n} \Pr(P_1\cup \cdots \cup P_\ell \cup P_T \subseteq P).
\label{eq:joint-moment}
	\end{equation}
We will find that $\E[ Y (X_1)_{x_1}\cdots (X_m)_{x_m}]$ is asymptotically dominated 
by the contribution from vertex-disjoint cycles. So we first evaluate
	\begin{align}
\Sigma^\ast(\xvec) &=
  \sum_{(P_1,\ldots, P_{\ell})\in\mathcal{S}^\ast(\xvec)}\,\,
	\sum_{P_T: G(P_T) \in \cT_n} \Pr(P_1\cup \cdots \cup P_\ell \cup P_T \subseteq P) 
\nonumber \\
	&= \sum_{(P_1,\ldots, P_{\ell})\in\mathcal{S}^\ast(\xvec)}\,\,
	\sum_{P_T: G(P_T) \in \cT_n} \frac{\abs{\{P\in\Omega_{n,r,s}: P_1\cup \cdots \cup P_\ell \cup P_T \subseteq P \}}}{\abs{\Omega_{n,r,s}}}.\label{Sigma-star}
	\end{align}
To perform this count, we condition on the intersections between $P_T$ and each of
the subpartitions $P_1,\ldots, P_\ell$. 
We will use the lexicographical ordering on $s$-subsets of vertices to define
a corresponding ordering on parts of $P_i$, by applying the lexicographical ordering
to the set (or multiset) of $s$ cells corresponding to the $s$ points in the part. 
First suppose that $G(P_i)$ is a loose $j$-cycle with $j\geq 2$.
The part of $P_i$ which is lexicographically-least will be the starting part of $P_i$,
and we fix a direction around $P_i$ such that the second part visited is lexicographically
smaller than the last part.  Then we define a binary sequence $I_i\in \{0,1\}^j$ corresponding
to $P_i$ as follows: starting from the first part of $P_i$, in the fixed direction,
if the $k$'th part of $P_i$ belongs to $P_i\cap P_T$ then the $k$'th element of $I_i$ is one; otherwise it is zero. 
All sequences in $\{0,1\}^j$ represent possible intersections, except for $(1,\dots, 1)$ because a tree contains no cycles. In the case that $G(P_i)$ is a 1-cycle then $I_i = (0)$.
Denote the set of all possible intersection sequences for a cycle of length $j$ by
	\[
	\cI_j = \{0,1\}^j \setminus \{(1,\dots, 1)\}
	\]
and define the Cartesian product
	\[
	\cI(\xvec) = \prod_{j=1}^m \cI_j^{x_j}.
	\]
For $I \in \cI_j$, let $U(I) \in \{1,\dots,j\}$ be the number of entries in $I$ which equal zero. Given $\Ivec = (I_1,\dots,I_\ell) \in \cI(\xvec)$, let $u_i = U(I_i)$for $i=1,\ldots, m$ and define $u = u(\Ivec) = u_1 + \dots + u_\ell$.

Given $(P_1,\ldots, P_\ell,P_T)$, write $\iota(P_1,\ldots, P_\ell,P_T) = \Ivec \in \cI(\xvec)$
for the corresponding $\ell$-tuple of intersection sequences.
We can rewrite $\Sigma^\ast(\xvec)$ as
	\[ \abs{\Omega_{n,r,s}}\, \Sigma^\ast(\xvec) \, = \sum_{\Ivec\in \cI(\xvec)} \sum_{\substack{(P_1,\ldots, P_{\ell},P_T):\\ (P_1,\ldots, P_{\ell})\in\mathcal{S}^\ast(\xvec), \\ G(P_T) \in \cT_n,\\\iota(P_1,\ldots, P_\ell,P_T) = \Ivec}} \
\abs{\{P\in\Omega_{n,r,s}: P_1\cup \cdots \cup P_\ell \cup P_T \subseteq P \}}.
	\]
For a given $\Ivec \in \cI(\xvec)$, we evaluate the inner sum using the following process:
	\begin{enumerate}
		\item[]{\bf Step 1}: Choose a sequence $(P_1,\ldots, P_{\ell})\in\mathcal{S}^*(\xvec)$.
	\item[] {\bf Step 2}: %Extend this to $P_1\cup \cdots \cup P_\ell \cup P_T$ consistent with $(I_1,\dots,I_\ell)$.
Choose $P_T$ with $G(P_T)\in\cT_n$ such that $\iota(P_1,\ldots, P_\ell,P_T) = \Ivec$. 
		\item[] {\bf Step 3}: Partition the remaining points arbitrarily.
	\end{enumerate}

Define the subpartition $Q = P_1 \cup \dots \cup P_\ell$. Writing $\abs{P'}$ for
the number of parts in a subpartition $P'$, we have
	\[
	\abs{Q} = \sum_{j=1}^m j \, x_j = \sum_{i=1}^\ell \abs{P_i}.
	\]
\begin{lemma}
\label{lem:s1}
Let $r,s\geq 2$ be integers such that $(r,s)\neq (2,2)$
and fix $\xvec= (x_1,\ldots, x_m)\in\mathbb{N}^m$.
The number of ways to choose a sequence of subpartitions 
$(P_1,\dots,P_\ell)$ in $\mathcal{S}^\ast (\xvec)$ is
\begin{equation}
\label{eq:s1}
s_1(\xvec)  \sim \left( \frac{(r-1)\, r^{s-1} n^{s-1}}{(s-2)!} \right)^{\abs{Q}} \prod_{j=1}^m \,\,\, \frac{1}{(2j)^{x_j}} .
\end{equation}
\end{lemma}

\begin{proof}
To begin, we claim that~\eqref{eq:s1} is true when $\ell=1$. 
%for a single loose $j$-cycle $C$. 
First suppose that $j\geq 2$.
We must show that the number of ways of selecting a subpartition $P$
which projects to a loose $j$-cycle $C$ is asymptotically equal to
\[
\frac{1}{2j} \left( \frac{(r-1)\, r^{s-1} n^{s-1}}{(s-2)!} \right)^j.
\] 
Recall that, in a loose $j$-cycle $C$, the $C$-external vertices have degree $2$ and the 
$C$-internal vertices have degree $1$. 
To specify a single loose $j$-cycle $C$, choose a sequence of $(s-1)j$ vertices in $(n)_{(s-1)j} \sim n^{(s-1)j}$ ways,
then divide by $2j((s-2)!)^j$. Here, division by $2j$ adjusts for direction and starting 
point (where a starting point is a $C$-external vertex),  and division by $((s-2)!)^j$ 
adjusts for the order of the $s-2$ $C$-internal vertices in each edge. To specify $P_C$, we choose two points for each $C$-external vertex and one point for each 
$C$-internal vertex in the configuration model,  in $(r(r-1))^j r^{(s-2)j}$ ways.
	Hence (\ref{eq:s1}) holds for a single loose $j$-cycle when $j\geq 2$.

When $j=1$, the number of non-loose 1-cycles is $O(n^{s-2})$ 
while the number of loose 1-cycles is $\Theta(n^{s-1})$.  
There are
$n\binom{n-1}{s-2}\sim \frac{n^{s-1}}{(s-2)!}$ ways to choose the vertices of a loose
1-cycle, where the first-chosen vertex is external, then there are $\binom{r}{2}\, r^{s-2}$ 
ways to choose points corresponding to these vertices.  
Multiplying these shows that (\ref{eq:s1}) also holds when $j=1$.

When $\ell > 1$, observe that this process can be iterated.  The only change is that
the next cycle must be disjoint from all previously-selected cycles, ruling out $O(1)$
vertices.  Hence the number of ways to select a sequence of $(s-1)j$ vertices for the next 
$j$-cycle is $\big(n - O(1)\big)^{(s-1)j}\sim n^{(s-1)j}$, and all remaining calculations are the same as above.  This
shows that (\ref{eq:s1}) holds in general.
\end{proof}

\bigskip

Now suppose that a sequence of subpartitions $(P_1, \dots, P_\ell)\in\mathcal{S}^\ast(\xvec)$ has been chosen. To perform Step~2, we construct an irregular configuration model
$\cP_{n',\xvec}$ from the points that are so far unused. 

There are $n - (s-1)\abs{Q}$ cells which are not involved in any of $(P_1,\dots,P_\ell)$.
Any cell which corresponds to an external vertex of some $G(P_i)$ has $r-2$ unused points,
and any cell which corresponds to an internal vertex of some $G(P_i)$ has $r-1$ unused points.
Recall that $\Ivec = (I_1,\dots,I_\ell)$ determines a collection of disjoint paths 
contained in the subpartitions $P_1,\dots,P_\ell$. This collection of paths will form the
intersection of $P_T$ and $P_1\cup\cdots P_{\ell}$.

For each such path with at least one part, 
collect all the unused points and combine them together into an \textit{irregular cell}. 
If such a path consists of $k$ part then it contains of $k+1$ cells with $r-2$ points unused,
and $(s-2)k$ cells with $r-1$ points unused.
Thus, the resulting irregular cell has 
$(k+1)(r-2) + (s-2)k(r-1) = (rs-r-s)k + r-2$ points. 
For each cell which corresponds to an external vertex of some $G(P_i)$, but which 
is not contained in the intersection $P_T\cap P_i$,
we also form an irregular cell with $r-2$ points.  Note, this matches the earlier formula
with $k=0$; we can think of these external vertices as a length-0 path in the intersection.
Indeed, these cells are exactly those which are contained in two parts of $P_i$ which
both correspond to a~0 in the intersection sequence $I_i$.

Recall that $u=u(\Ivec)=u_1+\dots+u_\ell$ where $u_i = U(I_i)$ is the number of zero entries
in $I_i$.
Then $u$ also equals the number of irregular cells identified so far, as the 
paths in the intersection $P_i\cap P_T$ (of length zero or more) are in one-to-one 
correspondence with the zero entries in $I_i$.
The $u$ irregular cells we have identified so far are called \textit{external irregular cells}.

Finally, for each cell which corresponds to an internal vertex of some $G(P_i)$, which is 
not involved in the intersection $P_T\cap P_i$,
we form an \textit{internal irregular cell} with $r-1$ points. There are $(s-2)u$ such cells. 
\bigskip

To summarise the properties of our irregular configuration model:
	\begin{itemize}
		\item The total number of cells is $n' = n - (s-1)\abs{Q} + (s-1)u$.
		\item There are $n - (s-1)\abs{Q}$ regular cells with $r$ points each.
		\item There are $(s-2)u$ internal irregular cells with $r-1$ points each.
		\item There are $u$ external irregular cells.   If an external irregular cell 
was collapsed from a path with $k$ parts then it contains $(rs-r-s)k + r-2$ points.
	\end{itemize}

The number of ways to complete Step~2 equals the number of ways of choosing a subpartition 
$P'$ in this irregular configuration model such that $G(P')$ is a spanning tree.
The projection $T' = G(P')$ of this partition corresponds exactly to a tree 
$T\in\cT_n$,  with the subpaths determined by $\Ivec$ contracted to single vertices.
	
For a sequence $I \in \cI_j$ and $k \ge 0$, let $\qk{I}$ be the number of paths of 
length $k$ in the intersection encoded by $I$. (Recall the length-0 paths correspond
to cells which belong to two parts in the $j$-cycle which are both encoded by 0 in $I$.)
By a slight abuse of notation, write $\qk{\Ivec} = \sum_{i=1}^\ell \qk{I_i}$.

The next result is proved in Section~\ref{sec:lem-s2}.

\bigskip

\begin{lemma}
\label{lem:s2}
Fix $\Ivec=(I_1,\dots,I_\ell) \in\cI(\xvec)$ and let $u=u(\Ivec)$.  
Fix a sequence of subpartitions $(P_1,\dots,P_\ell) \in \mathcal{S}^\ast (\xvec)$ and let 
$Q = P_1 \cup \dots \cup P_\ell$.
Then the number of ways to extend $Q$ to a subpartition $Q \cup P_T$
consistent with $\Ivec$,
such that $G(P_T)\in\mathcal{T}_n$, is 
\begin{align*}
   & s_2(\xvec,\Ivec)\\
  &\sim \frac{\sqrt{r-1}\, (s-1)^2\, ((s-1)!)^{\frac{1}{s-1}}}
 {(rs-r-s)^{\frac{3s-1}{2(s-1)}}\, n^2}\,
  \left(\frac{(rs-r-s)^{s+1}\, n^{s-1}}{(r-1)\, (s-1)^{s-1}\, (s-1)!}\right)^u\,
  \left(\frac{(s-2)!}{(r-1)\, r^{s-1}\, n^{s-1}}\right)^{\abs{Q}}\\
 & \qquad \qquad \times
   \, \left(\frac{r\, (r-1)^{r-1}\, (s-1)^{r-1}\, n}{e((s-1)!)^{\frac{1}{s-1}}
  \, (rs-r-s)^{\frac{rs-r-s}{s-1}}}\right)^n\,\,
\prod_{k=0}^{m-1}\, \left(k + \frac{r-2}{rs-r-s}\right)^{\qk{\Ivec}}.
\end{align*}
\end{lemma}

\bigskip
Finally, Step~3 completes the subpartition $P_T\cup Q$ to a partition $P\in \cP_{n,r,s}$.

\begin{lemma}
\label{lem:s3}
Given $\Ivec\in\cI(\xvec)$,
suppose that $(P_1,\ldots, P_\ell)\in\mathcal{S}^\ast(\xvec)$ is fixed, and $P_T$ is
a fixed subpartition with $\iota(P_1,\ldots, P_\ell,P_T) = \Ivec$.
Let $u=u(\Ivec)$.
The number of ways to complete Step~3 is 
\begin{align*}
	s_3(\xvec,\Ivec)
 & \sim  \frac{\sqrt{s}\, (rs-r-s)\, n}{(s-1)\, ((s-1)!)^{\frac{1}{s-1}}} \,
\left(\frac{(s-1)^{s-1}\, (s-1)!}{(rs-r-s)^{s-1}\, n^{s-1}} \right)^u\,
  \left(\frac{(rs-r-s)\, n}{e (s-1)\,((s-1)!)^{\frac{1}{s-1}}} \right)^{\frac{(rs-r-s)}{s} n}. 
\end{align*}
\end{lemma}

\begin{proof}
Out of the $rn$ points in the original configuration model, $2\abs{Q}$ points have been 
used for the external vertices in $Q = P_1\cup\cdots\cup P_\ell$ and 
$(s-2)\abs{Q}$ points have been used for the internal vertices of $Q$.  Finally,
$\frac{s(n'-1)}{s-1}$ points have been used to complete the subpartition $P_T$. 
So there are 
\[ rn - 2\abs{Q} - (s-2)\abs{Q} - \frac{s(n'-1)}{s-1} = 
  \frac{(rs-r-s)n}{s-1} - s\Big(u- \frac{1}{s-1}\Big)\]
points remaining. 
Hence, the number of ways to complete Step 3 is
	\begin{align}
	s_3(\xvec,\Ivec) &=  p\left(\frac{(rs-r-s)n}{s-1} - s\Big(u - \frac{1}{s-1}\Big)\right) \label{eq:s3-exact} 
	\end{align}
and applying Stirling's approximation completes the proof. 
\end{proof}

\bigskip

We use these expressions for $s_1$, $s_2$ and $s_3$ to prove the following result.

	\begin{lemma} \label{lemma:joint-moment}
Let $r,s \ge 2$ be fixed integers with $(r,s) \ne (2,2)$. 
For any fixed integer $m\geq 1$ and fixed sequence $(x_1,\dots,x_m)$ of non-negative integers,
		\[
		\frac{\E[Y(X_1)_{x_1} \dots (X_m)_{x_m}]}{\E Y} 
  \longrightarrow \prod_{j=1}^m \, (\lambda_j (1+ \zeta_j))^{x_j}
		\]
		as $n \to \infty$ along $\cN_{(r,s)}$, where for all $j\in\mathbb{Z}^+$,
		\[
\lambda_j = \frac{(r-1)^j(s-1)^j}{2j} \quad \text{and} \quad \zeta_j = \frac{\left(\frac{r}{r-1} - s +1 \right)^j -2}{(r-1)^j(s-1)^j}.
		\]
	\end{lemma}

\begin{proof}
Recall the definition of $\Sigma^*$ from (\ref{Sigma-star}).
By definition of $s_1(\xvec)$, $s_2(\xvec,\Ivec)$, $s_3(\xvec,\Ivec)$, we have
\[
 \Sigma^\ast(\xvec) = \sum_{\Ivec \in \cI(\xvec)} 
   \frac{s_1(\xvec) \, s_2(\xvec,\Ivec)\, s_3(\xvec,\Ivec)}{\abs{\Omega_{n,r,s}}\, \E Y}.
\]
Combining Lemmas~\ref{lem:s1}--\ref{lem:s3},
then dividing by (\ref{for-future-reference}) and cancelling leads to 
	\begin{align}
\Sigma^\ast(\xvec) %\frac{\E(YX_j)}{\E Y} 
&\sim \prod_{j=1}^m \frac{1}{(2j)^{x_j}} \sum_{\Ivec \in \cI(\xvec)} \left(\frac{(rs-r-s)^2}{r-1} \right)^{u} \,\,\, \prod_{k=0}^{m-1}  \left(k + \frac{r-2}{rs-r-s} \right)^{\qk{\Ivec}} \nonumber \\
	&= \prod_{j=1}^m \frac{1}{(2j)^{x_j}} \sum_{\Ivec \in \cI(\xvec)} \prod_{i=1}^\ell \left(\frac{(rs-r-s)^2}{r-1} \right)^{u_i} \,\,\, \prod_{k=0}^{m-1}  \left(k + \frac{r-2}{rs-r-s} \right)^{\qk{I_i}} \nonumber \\
	&= \prod_{j=1}^m \xi_j^{x_j},\label{eq:xis}
	\end{align}
where 
\[
\xi_j = \frac{1}{2j} \sum_{I \in \cI_j}\, \left(\frac{(rs-r-s)^2}{r-1} \right)^{U(I)} \,\,\,
  \prod_{k=0}^{j-1}  \left(k + \frac{r-2}{rs-r-s} \right)^{\qk{I}}.
\]
	We will compute this sum with the help of a generating function. Because $(1,\dots,1) \notin \cI_j$, we may identify a particular element in the sequence to be zero. By symmetry, we arbitrarily choose the last. Define the coefficients
	\[
	c_{j,\ell} = \sum_{\substack{I \in \cI_j: \\ U(I) = \ell \\ I_j = 0}} \mu^\ell \, \prod_{k=0}^{j-1} \left(k+\beta\right)^{\qk{I}}, %\quad \text{and} c_{0,\ell} = 0,
	\]
	where we let
	\[
	\mu = \frac{(rs-r-s)^2}{r-1} \quad \text{and} \quad \beta = \frac{r-2}{rs-r-s}
	\]
	for convenience. Now, $c_{j,\ell}$ fixes the number of zeros in $I$ to be $\ell$, and assumes that $I_j = 0$ (that is, the last entry of $I$ is zero). 
%There are $j$ positions in the sequence $I$, and / $\ell$ in $I$, of which $\ell$ zeros in $I$. 
Hence
	\begin{equation} \label{eq:EYX}
	\xi_j \sim \frac{1}{2j} \sum_{\ell=1}^j \frac{j\, c_{j,\ell}}{\ell} = \frac{1}{2} \sum_{\ell = 1}^j \frac{c_{j,\ell}}{\ell}.
	\end{equation}
	
	We now evaluate the coefficients $c_{j,\ell}$. 
Recall that $U(I)$ represents the number of zeros in the sequence $I$.
We have $c_{j,1} =\mu \left(j-1+\beta\right)$, because the only sequence with 1 zero and the last element zero is $(1,\dots,1, 0)$.
For $\ell \ge 2$, the sequence starts with $k$ ones followed by a zero, for
some $k\in \{0,\ldots, j -2\}$. Ranging over these possibilities gives
	\[
	c_{j,\ell} = \mu \sum_{k=0}^{j-2} \left(k + \beta\right)  c_{j-k-1,\ell-1}.
	\]
	To solve this, define the generating function
	\[
	F(x,y) = \sum_{j \ge 1} \sum_{ \ell \ge 1} c_{j,\ell}\, x^j y^\ell.
	\]
	By changing the order of summation and re-indexing, we have
	\begin{align*}
	F(x,y) - \sum_{j \ge 1} c_{j,1} x^j y &= \mu \sum_{j \ge 1} \sum_{\ell \ge 2} \sum_{k=0}^{j-2} \left(k + \beta\right) c_{j-k-1,\ell-1} \,x^j y^\ell \\
	&= \mu  \sum_{k \ge 0} \left(k + \beta\right) x^{k+1} y  \sum_{j \ge k+2} \sum_{\ell' \ge 1} c_{j-k-1,\ell'} \,x^{j-k-1} y^{\ell'} \\
	&= \mu  \sum_{k \ge 0} \left(k + \beta\right) x^{k+1} y \, F(x,y). 
	\end{align*}
	Thus, recalling that $c_{j,1} = \mu (j-1+\beta)$, we have
	\begin{align*}
	F(x,y) &= \mu \left[\sum_{j \ge 1} \left(j-1+\beta\right) x^j y +  \sum_{k \ge 0} \left(k + \beta\right) x^{k+1} y \, F(x,y)\right] \\
	&= \mu\,(F(x,y) + 1) \sum_{k \ge 0} \left(k + \beta\right) x^{k+1} y .
	\end{align*}
	Recall that by differentiating both sides of $(1-x)^{-1} = \sum_{k \ge 0} x^{k}$, we have
	\[
	\sum_{k\ge 0} k x^{k+1} = \frac{x^2}{(1-x)^2}.
	\]
	Hence if we define
	\begin{align*}
	f(x) = \mu \sum_{k \ge 0} \left(k + \beta\right) x^{k+1} = \mu\left(\frac{x^2}{(1-x)^2} + \frac{\beta x}{1-x}\right), 
	\end{align*}
	we have $F(x,y) = y \, f(x) (F(x,y) + 1)$ and thus
	\[
	F(x,y) = \frac{f(x)y}{1-f(x)y}.
	\]
	Now, going back to \eqref{eq:EYX}, we have
	\begin{align*}
	\xi_j &\sim \dfrac{1}{2} \sum_{\ell=1}^j \frac{c_{j,\ell}}{\ell} 
	= \dfrac{1}{2}[x^j]\, \sum_{\ell=1}^j \frac{1}{\ell}\, [y^{\ell-1}] \frac{f(x)}{1-f(x)y}.
	\end{align*}
	Applying the Taylor expansion of $(1-z)^{-1}$ and $\log(1-z)$, we have
	\[
	\xi_j \sim \dfrac{1}{2}[x^j] \sum_{\ell=1}^j \frac{1}{\ell} [y^{\ell-1}]  \left(f(x)\sum_{k=0}^\infty (f(x)y)^k \right) 
 = \dfrac{1}{2}[x^j] \sum_{\ell=1}^j \frac{f(x)^\ell}{\ell} 
= - \dfrac{1}{2}[x^j] \log(1-f(x)).
	\]
	Now,
	\[
	1 - f(x) = \frac{\left(1 - \left(\frac{r}{r-1} - s + 1\right) x\right) (1- (r-1)(s-1)x)}{(1-x)^2},
	\]
	so
	\begin{align*}
 	\xi_j &\sim \dfrac{1}{2} [x^j] \left(2\log(1-x) - \log\left(1 - \left(\frac{r}{r-1} - s + 1\right) x\right) - \log\left(1- (r-1)(s-1)x\right) \right) \\
	&= \dfrac{1}{2} [x^j] \sum_{k=1}^\infty \frac{-2x^k +  \left(\left(\frac{r}{r-1} - s + 1\right) x\right)^k + \left((r-1)(s-1)x \right)^k }{k} \\
	&= \frac{\left(\frac{r}{r-1} - s + 1 \right)^j + (r-1)^j (s-1)^j - 2}{2j} \\
	&= \lambda_j (1 + \zeta_j)
	\end{align*}
for $j=1,\ldots, m$.  Substituting this into (\ref{eq:xis}) implies that
\[ \Sigma^\ast(\xvec) \sim \prod_{j=1}^m \, (\lambda_j (1+ \zeta_j))^{x_j}.\]

To complete the proof, it remains to show that in (\ref{eq:joint-moment}), the sum over 
$\mathcal{S}(\xvec)\setminus \mathcal{S}^\ast(\xvec)$ is negligible. 
This is standard, but for completeness we sketch an argument.
We adapt Steps~1 to~3 as above.  There are $O(n^{(s-1)|Q|-1})$
ways to choose $(P_1,\ldots, P_{\ell})\in\mathcal{S}(\xvec)\setminus \mathcal{S}^\ast(\xvec)$,
as then $Q = P_1\cup \cdots\cup P_\ell$ involves at most $(s-1)|Q|-1$ distinct cells.
Now consider the number of ways to perform Steps~2 and~3, summed over all possibilities
for the intersection $Q\cap P_T$.  This is the number of ways to extend $Q$ to 
$Q\cup P_T$, where
$P_T$ corresponds to a spanning tree $T\in\cT_n$, and then extending $Q\cup P_T$
to a full partition.  This is very similar to the calculations performed
to evaluate $\E Y$, and the presence of $Q$ only changes these calculations by 
a constant factor.  Therefore, for a given $Q$, the total number of ways to perform
Steps~2 and~3, summed over all possible intersections, and then divided by $\E Y$,
is
\[ O(1)\,\, \frac{p(rn - s|Q|)}{p(rn)} = O(n^{-(s-1)|Q|}).\]
Multiplying this with the $O(n^{(s-1)|Q|-1})$  ways to complete Step~1,
we see that the sum over $(P_1,\ldots, P_{\ell})$ in (\ref{eq:joint-moment}), 
contributes $O(1/n) = o(1)$, as required.
\end{proof}

\bigskip

	We now show that condition (A3) holds, under fairly weak conditions on $(r,s)$.

	\begin{lemma} \label{lemma:a3}
	Fix integers $r,s\geq 2$ and $r$ such that 
\begin{equation} 
r\geq \begin{cases} 3 & \text{ if $s = 2$,}\\ 
                       2 &  \text{ if $s\in \{3,4\}$,}\\
                       s-1 &  \text{ if $s \geq  5$.}
 \end{cases}
\label{r-lower}
\end{equation}
%$r \ge 3$ and $s \ge 2$ be fixed integers where $r > \rho(s)$. 
Then 
	\[
		\exp\left(\sum_{j=1}^\infty \lambda_j \zeta_j^2 \right) = \frac{r^2 \sqrt{s-1}}{\sqrt{\left(r^2-r s+r+s-1\right) (r s-r-s)(r-1)}} < \infty.
		\]
	\end{lemma}

	\begin{proof}
		First, observe that
		\[ \frac{(\frac{r}{r-1}-s+1)^2}{(r-1)(s-1)},\qquad
		\frac{\frac{r}{r-1}-s+1}{(r-1)(s-1)},\qquad  \frac{1}{(r-1)(s-1)}\]
		are all less than~1 in absolute value. 
Next, we claim that
			\begin{equation} \label{eq:det-inequality}
			r^2-r s+r+s-1 >0. 
			\end{equation}
This condition is easily verified when $s\in \{2,3,4\}$ and $r$ belongs to the stated range.
When $s\geq 5$ we use the fact that
			\[
			r  \ge s-1 > \frac{1}{2} \left(\sqrt{s^2-6 s+5}+s-1\right),
			\]
			which implies \eqref{eq:det-inequality}.

Therefore, using the Taylor expansion of $-\log(1-z)$ we obtain
		\begin{align*}
		\sum_{j=1}^\infty \lambda_j \zeta_j^2 &= \dfrac{1}{2} \sum_{j=1}^\infty \frac{1}{j}\left( \left( \frac{\left(\frac{r}{r-1} - s +1 \right)^2}{(r-1)(s-1)} \right)^j - 4 \left( \frac{\frac{r}{r-1} - s + 1}{(r-1)(s-1)} \right)^j + 4 \left(\frac{1}{(r-1)(s-1)}\right)^j \right) \\
		&=-2 \log \left(1-\frac{1}{(r-1) (s-1)}\right)+2 \log \left(1-\frac{\frac{r}{r-1}-s+1}{(r-1) (s-1)}\right)\\
		& {} \hspace*{58mm}  -\dfrac{1}{2} \log \left(1-\frac{\left(\frac{r}{r-1}-s+1\right)^2}{(r-1) (s-1)}\right) \\
		&= -2 \log\left(\frac{rs-r-s}{(r-1)(s-1)} \right) + 2 \log\left(\frac{r(rs-r-s)}{(r-1)^2(s-1)} \right) \\
		& {} \hspace*{5cm} -\dfrac{1}{2} \log\left(\frac{(rs-r-s)(r^2-rs+r+s-1)}{(r-1)^3(s-1)} \right).
		\end{align*}
		Taking the exponential of both sides establishes the result.
	\end{proof}

Next, we investigate the parameters $\zeta_j(r,s)$.

	\begin{lemma} \label{lemma:zeta}
		Let $r,s \ge 2$ and  recall that for all fixed integers $j\geq 1$,  
		\[
		\zeta_j = \zeta_j(r,s) = \frac{\left(\frac{r}{r-1} - s +1 \right)^j -2}{(r-1)^j(s-1)^j}.
		\]
		Then 
		\begin{enumerate}
			\item[\emph{(i)}] $\zeta_j(2,2) = -1$ for $j \ge 1$;
			\item[\emph{(i)}] $\zeta_j(2,s) = -1$ for $s \ge 3$ and $j = 1$;
			%\item $\zeta_j(2,s) = -1$ for $s \ge 3$ and $j = 1$,
		\end{enumerate}
		In all other cases, $\zeta_j(r,s) > -1$.
	\end{lemma}
	\begin{proof}
		It is easy to check that (i) and (ii) hold. 
For (iii), note that $\zeta_j(r,s) > -1$ if and only if $f(r,s,j)>0$, where
		\[
		f(r,s,j) = \left( \frac{r}{r-1} - s + 1\right)^j - 2 + (r-1)^j (s-1)^j.
		\]
		If $r=2$ and  $s \ge 3$ then $f(2,s,j) = (s-1)^j -(s-3)^j - 2$, 
which increases with $j$. So 
		\[
		f(2,s,j) \ge f(2,s,2) = 4s - 10 > 0
		\]	
		as $s \ge 3$. If $s =2$, then $\left( \frac{r}{r-1} - s + 1\right)^j > 0$, so $f(r,2,j) >  (r-1)^j -2 \ge 0$ as $r \ge 3$ and $j\geq 1$.
		
\medskip
		It remains to show that $f(r,s,j) > 0$ when $r \ge 3$, $s \ge 3$ and $j \ge 1$. 
If $j$ is even then $\left( \frac{r}{r-1} - s + 1\right)^j \ge 0$, so
		\[
		f(r,s,j) \ge (r-1)^j (s-1)^j - 2 > 2^{2j} - 2 > 0.
		\]
		Now suppose that $j$ is odd.  Since $r \ge 3$ and $s \ge 3$, 
we have $\frac{r}{r-1} -s +1 < 0$. Thus
		\begin{align*}
		f(r,s,j) &= (r-1)^j (s-1)^j - \left(s-1-\frac{r}{r-1}\right)^j - 2 \\
		&> \left((r-1)^j - 1 \right)(s-1)^j - 2 \\
		&\ge 2^j - 2 \ge 0.
  	\end{align*}
This completes the proof.
	\end{proof}

	\section{Second moment} \label{chap:2ndmoment}
	
	So far, conditions (A1)--(A3) of Theorem \ref{thm:subgraph} have been verified. 
	In this section we will assume that $r,s\geq 2$ are fixed integers such that
$r > \rho(s)$ when $s\geq 5$, and that $(r,s)\neq (2,2)$ when $s\in \{ 2, 3, 4\}$.
With this assumption, we will 
obtain an asymptotic expression for the second moment, verifying condition (A4).
The following identity of Chu~\cite{chu1986extension}, which generalises Jensen's identity,
will be useful.
	
	\begin{lemma}[\cite{chu1986extension}] \label{lemma:jensen}
		Let $m$, $b$ be positive integers, $x_1,\ldots, x_b$ and $z$ be complex numbers, and define $\binom{-1}{0} = 1$. Then
		\[
		\sum_{\substack{(k_1,\ldots, k_b)\in\mathbb{N}^b\\k_1 + \dots + k_b = m}} \, \prod_{i=1}^b \,
		\binom{x_i+k_iz}{k_i} = \sum_{k=0}^m \binom{k+b-2}{k} \binom{x_1 + \dots + x_b + mz - k}{m-k} \,z^k.
		\]
	\end{lemma}

	We write
	\[
	\abs{\Omega_{n,r,s}} \E Y^2 = \sum_{(P_{T_1},P_{T_2})}  \abs{\{ P \in \Omega_{n,r,s} : P_{T_1} \cup P_{T_2} \subseteq P \}}
	\]
	where the sum is over all pairs $(P_{T_1}, P_{T_2})$ such that $G(P_{T_1}) = T_2$ and $G(P_{T_2}) = T_2$, for some spanning trees $T_1, T_2$. 
	
	We perform this count by conditioning on the intersection between $P_{T_1}$ and $P_{T_2}$, which will correspond to a union of disjoint trees. 
Let $b \in \left\{1 + (s-1)\ell: \ell = 0, 1, \dots, \frac{n-1}{s-1}  \right\}$ be the number of connected components in this intersection (we can show $b$ must be of this form by adding up the number of vertices in each connected component). We break up the process into the following steps:
	\begin{enumerate}
		\item Choose a partition $\nu = (\nu_1, \dots, \nu_b)$ of $n$, where $\nu_i > 0$, $s-1 \mid \nu_i - 1$ and $\sum_{i =1}^b \nu_i = n$. Here, $\nu_i$ represents the number of vertices in the $i$'th connected component. (Later, we will divide by $b!$ to account for the assumption that the connected components are labelled).
		\item Choose a partition of the $n$ vertices into $b$ groups, where the size of the $i$'th group is $\nu_i$.
		\item In each group, choose a spanning tree on that group and a subpartition that projects to that tree. 
	\end{enumerate}
	We then collapse the unused points in each group to an irregular cell. The $i$'th irregular cell will have $r \nu_i - \frac{s(\nu_i - 1)}{s-1} = \frac{rs-r-s}{s-1}\nu_i + \frac{s}{s-1} $ points. In this irregular configuration model, we wish to partition two part-disjoint spanning trees $T_1'$ and $T_2'$, which will extend to $T_1$ and $T_2$.
	\begin{enumerate} \setcounter{enumi}{3}
		\item Choose $\boldsymbol{\delta}^{(1)}, \boldsymbol{\delta}^{(2)} \in \N^{b}$, the degree sequence of $T_1'$ and $T_2'$ respectively, such that, for all $i$,
		\[
		\delta_i^{(1)}, \delta_i^{(2)} \ge 1, \quad \sum_{i =1}^b \delta_i^{(1)} = \sum_{i =1}^b \delta_i^{(2)} = \frac{s(b-1)}{s-1}, \quad \text{and} \quad \delta_i^{(1)} + \delta_i^{(2)} \le \frac{rs-r-s}{s-1}\nu_i + \frac{s}{s-1}.
		\]
		\item Choose trees $T_1'$, $T_2'$ consistent with $\boldsymbol{\delta}^{(1)}$ and $\boldsymbol{\delta}^{(2)}$.
		\item Choose $P_{T_1}$ and $P_{T_2}$ such that there are no parts in common.
		\item Partition remaining points.
	\end{enumerate}
	Then $\abs{\Omega_{n,r,s}} \E Y^2$ is equal to the number of ways to complete the above process, summed over all $b$.
	
Let
	\[
	\mathcal{S}_1 (b) = \left\{\nu \in \left\{1 + (s-1)\ell: \ell =0, 1, \dots, \frac{n-1}{s-1} \right\}^b : \sum_{i =1}^b \nu_i = n\right\}
	\]
	be the set of possible sequences $\nu$ from Step 1. 
	The number of ways to complete Step 2 is
	\[
	s_2  = \binom{n}{\nu_1, \dots, \nu_b} = n! \prod_{i=1}^b \frac{1}{\nu_i!}.
	\]
	By \eqref{number-of-tree-subpartitions}, the number of ways to complete Step 3 is
	\[
	s_3 = \frac{(s-1)^b \,r^n}{((s-1)!)^{\frac{n-b}{s-1}}} \, \prod_{i=1}^b (\nu_i -2)! \binom{(r-1)\nu_i}{\frac{\nu_i - s}{s-1}} = \frac{(s-1)^b\, r^n}{((s-1)!)^{\frac{n-b}{s-1}}} \, \prod_{i=1}^b  \frac{(\nu_i -2)! ((r-1)\nu_i)!}{\left(\frac{\nu_i-s}{s-1} \right)! \left(\frac{rs-r-s}{s-1} \nu_i + \frac{s}{s-1} \right)!}.
	\]
Let 
	\begin{align*}
	\mathcal{S}_4 (\nu) = \bigg\{ (\boldsymbol{\eta}^{(1)},\boldsymbol{\eta}^{(2)},\boldsymbol{\eta}^{(3)}) \in (\N^b)^3: \quad &\eta_i^{(1)} + \eta_i^{(2)} + \eta_i^{(3)} = \frac{rs-r-s}{s-1} \nu_i - \frac{s-2}{s-1}, \\
	&\sum_{i =1}^b \eta_i^{(1)} = \sum_{i =1}^b \eta_i^{(2)} = \frac{b-s}{s-1} \bigg\}
	\end{align*}
	be the set of sequences arising from Step 4.
	By (\ref{degree-sequence-hypertree}), the number of ways to complete Step 5 is
	\[
	s_5 = \binom{\frac{b-1}{s-1} - 1}{\delta_1^{(1)} - 1, \dots, \delta_b^{(1)}-1} \binom{\frac{b-1}{s-1} - 1}{\delta_1^{(2)} - 1, \dots, \delta_b^{(2)}-1} \left( \frac{(b-1)!}{\left(\frac{b-1}{s-1}\right)! ((s-1)!)^{\frac{b-1}{s-1}}} \right)^2.
	\]
	As each cell in this irregular configuration model has $\frac{rs-r-s}{s-1}\nu_i + \frac{s}{s-1}$ points, the number of ways to complete Step 6 is
	\[
	s_6 = \prod_{i=1}^b \left( \frac{rs-r-s}{s-1}\nu_i + \frac{s}{s-1} \right)_{\delta_i^{(1)}+\delta_i^{(2)}} = \prod_{i =1}^b \frac{\left(\frac{rs-r-s}{s-1} \nu_i + \frac{s}{s-1} \right)!}{\left(\frac{rs-r-s}{s-1} \nu_i + \frac{s}{s-1} - \delta_i^{(1)} - \delta_i^{(2)} \right)!}.
	\]
	There are
	\[
	\sum_{i =1}^b r \nu_i - \frac{s(\nu_i -1)}{s-1} - \delta_i^{(1)} - \delta_i^{(2)} =\frac{(rs-r-s)n}{s-1} - \frac{s(b-2)}{s-1} 
	\]
	points remaining, so the number of ways to complete Step 7 is
	\[
	s_7 = p\left(\frac{(rs-r-s)n}{s-1} - \frac{s(b-2)}{s-1} \right) = \frac{\left( \frac{(rs-r-s)n}{s-1} - \frac{s(b-2)}{s-1}   \right)!}{\left( \frac{rs-r-s}{s(s-1)} n - \frac{b-2}{s-1} \right)! (s!)^{\frac{rs-r-s}{s(s-1)} n - \frac{b-2}{s-1}}}.
	\]
	
	It is convenient to work with nonnegative variables, so we let 
	\[
	\eta_i^{(1)} = \delta_i^{(1)} - 1, \quad \eta_i^{(2)} = \delta_i^{(2)} - 1, \quad \eta_i^{(3)} = \frac{rs-r-s}{s-1} \nu_i - \frac{s-2}{s-1} - \eta_i^{(1)} - \eta_i^{(2)},
	\]
	for $i = 1, \dots, b$. 

 Combining everything, and dividing by $b!$ as promised earlier, we have
	\begin{align*}
	\abs{\Omega_{n,r,s}} \E Y^2 &= \sum_{\substack{b=1 \\ s-1 \mid b-1}}^n \frac{1}{b!} \sum_{ \nu \in \mathcal{S}_1(b)} s_2 s_3 \sum_{(\boldsymbol{\eta}^{(1)}, \boldsymbol{\eta}^{(2)}, \boldsymbol{\eta}^{(3)}) \in \mathcal{S}_4(\nu)} s_5 s_6 s_7 \\
	&= \sum_{\substack{b=1 \\ s-1 \mid b-1}}^n \frac{n! \, r^n (s-1)^b ((b-1)!)^2}{b! ((s-1)!)^{\frac{n+b-2}{s-1}}\left( \frac{rs-r-s}{s(s-1)}n - \frac{b-2}{s-1} \right)! \left(\left(\frac{b-1}{s-1} \right)! \right)^2 (s!)^{\frac{rs-r-s}{s(s-1)}n - \frac{b-2}{s-1}} } \\
	& \quad \times \sum_{ \nu \in \mathcal{S}_1(b)} \left(\prod_{i=1}^b \frac{((r-1)\nu_i)!}{\nu_i (\nu_i -1)\left(\frac{\nu_i - s}{s-1} \right)!} \right) \\
	&\quad \times  \sum_{(\boldsymbol{\eta}^{(1)}, \boldsymbol{\eta}^{(2)}, \boldsymbol{\eta}^{(3)}) \in \mathcal{S}_4(\nu)} \binom{\frac{b-1}{s-1} - 1}{\eta_1^{(1)}, \dots, \eta_b^{(1)}} \binom{\frac{b-1}{s-1} - 1}{\eta_1^{(2)}, \dots, \eta_b^{(2)}} \binom{\frac{(rs-r-s)n}{s-1} - \frac{s(b-2)}{s-1}}{\eta_1^{(3)}, \dots, \eta_b^{(3)}}.
	\end{align*}
	Now we compute the sum over $\cS_4(\nu)$ through the use of generating functions:
	\begin{align*}
	&\phantom{=} \sum_{(\boldsymbol{\eta}^{(1)}, \boldsymbol{\eta}^{(2)}, \boldsymbol{\eta}^{(3)}) \in \mathcal{S}_4(\nu)} \binom{\frac{b-1}{s-1} - 1}{\eta_1^{(1)}, \dots, \eta_b^{(1)}} \binom{\frac{b-1}{s-1} - 1}{\eta_1^{(2)}, \dots, \eta_b^{(2)}} \binom{\frac{(rs-r-s)n}{s-1} - \frac{s(b-2)}{s-1}}{\eta_1^{(3)}, \dots, \eta_b^{(3)}} \\
	&= \sum_{(\boldsymbol{\eta}^{(1)}, \boldsymbol{\eta}^{(2)}, \boldsymbol{\eta}^{(3)}) \in \mathcal{S}_4(\nu)} 
\left( [z_1^{\eta_1^{(1)}} \cdots z_b^{\eta_b^{(1)}} ] \bigg(\sum_{i =1}^b z_i \bigg)^{\frac{b-s}{s-1}} \right)
\left( [z_1^{\eta_1^{(2)}} \cdots z_b^{\eta_b^{(2)}} ] \bigg(\sum_{i =1}^b z_i \bigg)^{\frac{b-s}{s-1}}\right) \\
	&\phantom{=\sum_{(\boldsymbol{\eta}^{(1)}, \boldsymbol{\eta}^{(2)}, \boldsymbol{\eta}^{(3)}) \in \mathcal{S}_4(\nu)}}\; \left( [z_1^{\eta_1^{(3)}} \cdots z_b^{\eta_b^{(3)}} ] \bigg(\sum_{i =1}^b z_i \bigg)^{\frac{(rs-r-s)n}{s-1} - \frac{s(b-2)}{s-1}}\right) \\
	&= [z_1^{\frac{rs-r-s}{s-1} \nu_1 - \frac{s-2}{s-1}} \cdots z_b^{\frac{rs-r-s}{s-1} \nu_b - \frac{s-2}{s-1}} ] \left(\sum_{i =1}^b z_i \right)^{\frac{(rs-r-s)n}{s-1} - \frac{b(s-2)}{s-1} } \\
	&= \binom{\frac{(rs-r-s)n}{s-1} - \frac{b(s-2)}{s-1}}{\frac{(rs-r-s)\nu_1}{s-1} - \frac{s-2}{s-1}, \dots, \frac{(rs-r-s)\nu_b}{s-1} - \frac{s-2}{s-1}}.
	\end{align*}
	So
	\begin{align*}
	\abs{\Omega_{n,r,s}} \E Y^2 &= \sum_{\substack{b=1 \\ s-1 \mid b-1}}^n \frac{n! \, r^n (r-1)^b ((b-1)!)^2\left(\frac{(rs-r-s)n}{s-1} - \frac{b(s-2)}{s-1} \right)!}{b! ((s-1)!)^{\frac{n+b-2}{s-1}}\left( \frac{rs-r-s}{s(s-1)}n - \frac{b-2}{s-1} \right)! \left(\left(\frac{b-1}{s-1} \right)! \right)^2 (s!)^{\frac{rs-r-s}{s(s-1)}n - \frac{b-2}{s-1}} } \\
	&\qquad \times \sum_{ \nu \in \mathcal{S}_1(b)} \prod_{i=1}^b \binom{(r-1)\nu_i - 1}{\frac{\nu_i -1}{s-1}} 
	\end{align*}
	Note that the summand is equal to $\abs{\Omega_{n,r,s}} \E Y$ when $b = 1$. 
	For $b \ge 2$, we let $k_i = \frac{\nu_i -1 }{s-1}$ and use Lemma \ref{lemma:jensen}, to see that
	\begin{align*}
	\sum_{ \nu \in \mathcal{S}_1(b)} \prod_{i=1}^b  \binom{(r-1)\nu_i - 1}{\frac{\nu_i -1}{s-1}} &= \sum_{\substack{k_1 + \dots + k_b = \frac{n-b}{s-1}\\k_i \ge 0 }} \prod_{i=1}^b \binom{(r-1)(s-1)k_i + r-2}{k_i} \\
	&= \sum_{k=0}^{\frac{n-b}{s-1}} \binom{k+b-2}{k} \binom{(r-1)n -b-k}{\frac{n-b}{s-1}-k} (r-1)^k.
	\end{align*}
	Define
	\begin{align}
	K &= \{(\alpha,\beta): \alpha, \beta \ge 0, \quad (s-1) \alpha + \beta \le 1 \}, \label{eq:K-hypergraphs} \\
	\mathcal{L} &=  \Z \times (s-1) \Z \label{eq:L-hypergraphs}
	\end{align}
	and let $K^{\circ}$ denote the interior of $K$.
	Thus, dividing through by the expression in \eqref{eq:partitions-hypergraph}, we have
	\begin{equation}
	\label{bigsum}
	\E Y^2 = \E Y + \sum_{(k,b) \in (\mathcal{L}+(0,1)) \cap nK} a_n(k,b),
	\end{equation}
	where 
	\begin{align*}
	a_n(k,b) &=\begin{cases}  0 & \text{for $b \le 1$,} \\
	\dfrac{ r^n (b-1)(r-1)^{k+b} \, s^{\frac{n+b-2}{s-1}} \, (k+b-2)! \, ((r-1)n-k-b)! \, (rn/s)! \, n!}{b \, k! \,  \left(\left(\frac{b-1}{s-1} \right)! \right)^2 \left( \frac{rs-r-s}{s(s-1)}n - \frac{b-2}{s-1} \right)! \,   \left(\frac{n-(s-1)k-b}{s-1} \right)! \, (rn)! } & \text{otherwise.} \end{cases}.
	\end{align*}
	We now wish to apply Laplace's method to compute the asymptotic summation of this expression.
	
	Greenhill, Janson and Ruci{\'n}ski \cite{randomlifts} proved a version of Laplace's method for asymptotic summation, tailored for the small subgraph conditioning method. 
	We refer to \cite{randomlifts} for precise definitions. 
	
	\begin{lemma}[\protect{\cite[Lemma 6.3]{randomlifts}}] \label{lemma:laplace}
		Suppose the following:
		\begin{enumerate}
			\item[\emph{(i)}] $\mathcal{L} \subset \R^m$ is a lattice with full rank $m$.
			\item[\emph{(ii)}] $K \subset \R^m$ is a compact convex set with non-empty interior.
			\item[\emph{(iii)}] $\varphi: K \to \R$ is a continuous function with a unique maximum at some interior point $x_0 \in K^\circ$.
			\item[\emph{(iv)}] $\varphi$ is a twice continuously differentiable in a neighbourhood of $x_0$ and the Hessian $H_0:=D^2 \varphi(x_0)$ is strictly negative definite.
			\item[\emph{(v)}] $\psi: K_1 \to \R$ is a continuous function on some neighbourhood $K_1 \subset K$ of $x_0$ with $\psi(x_0) > 0$.
			\item[\emph{(vi)}] For each positive integer $n$ there is a vector $\ell_n \in \R^m$.
			\item[\emph{(vii)}] For each positive integer $n$ there is a positive real number $b_n$ and a function $a_n: (\mathcal{L}+\ell_n) \cap nK \to \R$ such that, as $n \to \infty$,
			\begin{align}
			&a_n(\ell) = O(b_ne^{n\varphi(\ell/n) + o(n)}), & \ell \in (\mathcal{L}+\ell_n) \cap nK, \label{eq:lemma-1st-condition}
			\intertext{and} 
			&a_n(\ell) = b_n(\psi(\ell/n) + o(1))e^{n\varphi(\ell/n)}, & \ell \in (\mathcal{L}+\ell_n) \cap nK_1, \label{eq:lemma-2nd-condition}
			\end{align}
			uniformly for $\ell$ in the indicated sets.
		\end{enumerate}
		Then, as $n \to \infty$,
		\[
		\sum_{(\mathcal{L} + \ell_n) \cap nK} a_n(\ell) \sim \frac{(2\pi)^{m/2} \psi(x_0)}{\det(\mathcal{L}) \det(-H_0)^{1/2}}\, b_n n^{m/2} e^{n \varphi(x_0)}.
		\]
	\end{lemma}
	
	To apply this lemma, we define
	\begin{align}
	b_n &= \frac{(s-1)^2}{2\pi n^3} \left(\frac{(s-1)^{r/s}}{r^{\frac{rs-r-s}{s}}} \right)^n, \nonumber \\
	\psi(\alpha,\beta)&= \frac{(r-1-\alpha-\beta)^{1/2}}{(\alpha+\beta)^{3/2} (rs-r-s(1+\beta))^{\frac{1}{2}+\frac{2}{s-1}} \beta^{1-\frac{2}{s-1}} \alpha^{1/2} (1-\beta-(s-1)\alpha)^{1/2}}, \nonumber \\
	\varphi(\alpha,\beta) &= (\alpha +\beta ) \log (r-1)+g(\alpha +\beta )+g(r-1-\alpha -\beta) -\frac{2}{s-1}\, g(\beta) - g(\alpha) \\
	&\qquad - \frac{1}{s(s-1)} \, g(r s-r-s-s\beta) - \frac{1}{s-1}\, g(1-(s-1)\alpha - \beta),
	\label{phi-def}
	\end{align}
	where $g(x) = x \log x$ for $x > 0$ and $g(0) = 0$. 
	The following result, proved in Section~\ref{s:global-max}, gives critical information about the function $\varphi$.
	
	\begin{lemma} \label{lem:global-max}
	Assume that $r,s\geq 2$ are fixed integers such that
$r > \rho(s)$ when $s\geq 5$, and that $(r,s)\neq (2,2)$ when $s\in \{ 2, 3, 4\}$.
	%For fixed integers $r \ge 3$, $s \ge 3$, where $r > \rho(s)$, 
The unique global maximum of $\varphi$ over $K$ 
		occurs at the point $(\alpha_0,\beta_0)$ where 
		\[
		\alpha_0 = \frac{1}{r(s-1)}, \quad \beta_0 = \frac{rs-r-s}{r(s-1)}.
		\]
		The maximum value of $\varphi$ over this domain equals
		\[
		\varphi(\alpha_0,\beta_0) = 2(r-1) \log(r-1) - \frac{2(rs-r-s)}{s(s-1)} \log(rs-r-s) + \frac{r}{s} \log(s-1) - \frac{rs-r-s}{s} \log r.
		\]
		Let $H_0$ be the Hessian of $\varphi$ evaluated at the point $(\alpha_0,\beta_0)$. Then $H_0$ is strictly negative definite and
		\[
		\det(-H_0) = \frac{r^3 (s-1)^2 \left(r^2-r s+r+s-1\right)}{(r-1)^2 (r s-r-s)}.
		\]
	\end{lemma}
	
	Then Lemma~\ref{lem:global-max} implies the following.
	
	\begin{lemma} \label{lemma:2nd-moment}
	Assume that $r,s\geq 2$ are fixed integers such that
$r > \rho(s)$ when $s\geq 5$, and that $(r,s)\neq (2,2)$ when $s\in \{ 2, 3, 4\}$.
		Then as $n \to \infty$ along $\cN_{(r,s)}$,
		\[
		\frac{\E Y^2}{(\E Y)^2} \sim \frac{r^2 \sqrt{s-1}}{\sqrt{\left(r^2-r s+r+s-1\right) (r s-r-s)(r-1)}}.
		\]
	\end{lemma}
	
	\begin{proof}
		We apply Lemma \ref{lemma:laplace} to compute the sum in \eqref{bigsum}. The first six conditions of the lemma hold:
		Now the conditions of Lemma \ref{lemma:laplace} hold:
		\begin{enumerate}
			\item[(i)] We defined $\mathcal{L} = \Z \times (s-1)\Z$, a lattice with rank $m = 2$ and $\det(\mathcal{L}) = s-1$.
			\item[(ii)] The domain $K$, defined in \eqref{eq:K-hypergraphs}, is compact and convex with a non-empty interior.
			\item[(iii)] The function $\varphi: K \to \R$ is a continuous function with a unique global maximum $(\alpha_0,\beta_0)$,  by Lemma~\ref{lem:global-max} is true.
			\item[(iv)] The function $\varphi:K \to \R$ is twice differentiable in the interior of $K$, with a strictly negative definite Hessian, by Lemma~\ref{lem:global-max} is true.
			\item[(v)] Let $K_1$ be the open ball around $(\alpha_0,\beta_0)$ of sufficiently small radius, ensuring that $K_1 \subset K$. 
			The function $\psi: K_1 \to \R$ is a continuous function with 
			\[
			\psi(\alpha_0,\beta_0) = \frac{r^{7/2}(s-1)^{5/2}}{\sqrt{r-1} (rs-r-s)^{\frac{2s}{s-1}}} > 0.
			\]
			\item[(vi)] Let $\ell_n$ be $(0,1)$ for each $n$.
			\item[(vii)] This condition is verified by applying Stirling's approximation.
			%Lemma \ref{lemma:stirling-strong} and Corollary \ref{lemma:stirling} verify this condition. We omit the details.
		\end{enumerate}
		Thus, we can apply Lemma \ref{lemma:laplace} to see that
		\begin{align*}
		\E Y^2 &\sim \frac{2\pi \, \psi(\alpha_0,\beta_0)}{\det(\mathcal{L}) \det(-H_0)^{1/2}} \,\, b_n \, n \, e^{n\ \varphi(\alpha_0,\beta_0)} \\
		&= \frac{r^2 \sqrt{r-1}\,  (s-1)^{5/2}\, (rs-r-s)^{\frac{1}{2} - \frac{2s}{s-1}}}{n^2 \, \sqrt{r^2 -rs+r+s-1}} \left(\frac{(s-1)^{r/s}\, (r-1)^{r-1}}{r^{\frac{rs-r-s}{s}} \, (rs-r-s)^{\frac{rs-r-s}{s(s-1)}}} \right)^{2n}.
		\end{align*}
Dividing by $(\E Y)^2$, using the expression from Lemma~\ref{lem:EY-hypergraph},
completes the proof.
	\end{proof}

	\section{Threshold analysis}\label{s:threshold}
	
	Define the logarithm of the base of the exponential factor as a function:
	\[
	L(r,s) = L_s (r) = \frac{r}{s} \log(s-1) + (r-1) \log(r-1) - \frac{rs-r-s}{s} \log r - \frac{rs-r-s}{s(s-1)} \log(rs-r-s),
	\]
	treating $r$ as a continuous variable and $s\geq 2$ as a fixed positive integer.
	We restrict to $r \ge 2$, or $r \ge 3$ when $s = 2$. 
	We want to determine when $L_s(r) > 0$ (which implies that $\E Y \to\infty$) and when $L_s(r) \leq 0$ 
	(which implies that $\E Y = o(1)$. 
	
	The following can be checked using elementary calculus.  
	
	\begin{lemma}
 For $s \in \{2,3,4\}$, $L_s(r) > 0$ for $r \in (2,\infty)$. Furthermore, for $s \in \{3,4\}$, $L_s(2) > 0$.
		\label{small-values}
	\end{lemma}
	
	The situation is quite different when $s\geq 5$, as Lemma~\ref{lem:unique-threshold} shows.
	The following inequality will be useful: For all $b \ge 1$ and $\abs{a} < b$, 
	\begin{equation} \label{in:trick}
	\left(1+\frac{a}{b} \right)^b \ge e^a \left(1 - \frac{a^2}{b} \right).
	\end{equation}
	(See for example~\cite[p.\ 435]{motwani2010randomized}.)
The proof of Lemma~\ref{lem:unique-threshold} is presented in Section~\ref{s:unique-threshold}.

	\begin{lemma} \label{lem:unique-threshold}
For a fixed integer $s \ge 5$, there exists a unique real number $\rho(s) > 2$ such that $L(\rho(s),s) = 0$,
		\[
		L(r,s) < 0 \quad \text{ for } r \in [2,\rho(s)) \quad \text{ and } \quad L(r,s) > 0 \text{ for } r \in (\rho(s),\infty).
		\]
		Furthermore, if $s \ge 6$, then $\rho(s) > s$.
	\end{lemma}

The next lemma proves that $\rho(s)$ is exponential in $s$ and lies strictly within an interval of unit width, and gives an asymptotic expression for
$\rho(s)$ with exponentially small error, as $s\to\infty$.  
The proof is given in Section~\ref{s:lem-bounds}.

	\begin{lemma} \label{lem:bounds}
For a fixed integer $s \ge 5$, let $\rho(s)$ be the unique real number such that $L(\rho(s),s) = 0$. Then
		\[
		\frac{e^{s-2}}{s-1}-\frac{s-1}{2} < \rho(s) < \frac{e^{s-2}}{s-1} - \frac{s-3}{2}.
		\]
Furthermore, as $s\to \infty$,
\[ \rho(s) = \frac{e^{s-2}}{s-1}- \frac{s^2-3s+1}{2(s-1)} + O_s(s^4\, e^{-s}).\]
	\end{lemma}
\bigskip
	
	We can now prove our main result, Theorem \ref{thm:threshold}, which gives us a threshold result for the existence of a spanning tree in $\cG_{n,r,s}$ when $s\geq 5$.
	
	\begin{proof}[Proof of Theorem \ref{thm:threshold}]
		Lemma \ref{lem:unique-threshold} proves the existence and uniqueness of $\rho(s)$ for $s \ge 5$, while Lemma \ref{lem:bounds} proves the upper and lower bounds on $\rho(s)$, and verifies the given asymptotic expression for $\rho(s)$.
		
		For fixed $s \ge 2$ and $r \le \rho(s)$ (but $(r,s) \ne (2,2)$), the base of the exponential factor in Lemma~\ref{lem:EY-hypergraph} is less than $1$, and thus we have $\Pr(Y \ge 1) \le \E Y \to 0$. By Lemma~\ref{lemma:translate}, we conclude that $\Pr(Y_{\cG} \ge 1) \to 0$.
		
		Now suppose that $s \ge 6$ and $r > \rho(s)$, or $s\in \{2,3,4\}$ and $(r,s)\neq (2,2)$. We verify the conditions of Theorem \ref{thm:subgraph}. 
		Cooper, Frieze, Molloy and Reed~\cite{cooper1996perfect} proved that condition~(1) holds with $\lambda_j$ defined in 
\eqref{eq:lambda}. 
Condition (A2) holds by Lemma~\ref{lemma:joint-moment},
			condition (A3) holds by Lemma \ref{lemma:a3}, 
			and condition (A4) holds by Lemma \ref{lemma:2nd-moment}. 
By our assumptions on $r,s$, we have $\zeta_1 > -1$ for all $j \ge 2$,
by Lemma \ref{lemma:zeta}. Hence $\widehat{Y} > 0$ a.a.s., where $\widehat{Y}$ is the random variable
obtained from $Y$ by conditioning on the event $X_1=0$. 
Then Lemma~\ref{lemma:translate} shows that $Y_{\cG} > 0$ a.a.s., as required. 
	\end{proof}

		Finally, we provide the proof for Theorem \ref{thm:distribution}, which gives the asymptotic distribution of the number of spanning trees in $\cG_{n,r,s}$.

	\begin{proof}[Proof of Theorem \ref{thm:distribution}]
If $\Pr(Y_{\cG}>0)\to 0$ then the result is immediate, as the asymptotic
value of $Y_{\cG}$ is zero with probability 1.  For the remainder of the
proof, suppose that $\Pr(Y_{\cG}>0)\to 1$.

		We showed in the proof of Theorem \ref{thm:threshold} that conditions (A1)--(A4) of Theorem \ref{thm:subgraph} hold for $Y$. This gives us
			\[
			\frac{Y}{\E Y} \,\,\, \overset{d}{\longrightarrow} \,\,\, \prod_{j=1}^\infty (1+\zeta_j)^{Z_j} e^{-\lambda_j \zeta_j}.
			\]	
		For $s=2$, the asymptotic distribution of $Y_\cG$ is obtained from the asymptotic distribution of $Y$ by conditioning on $X_1 = X_2 = 0$.
		
		For $s \ge 3$, let $\widehat{Y}$ be the random variable obtained from $Y$ by conditioning on the event $X_1 = 0$. We apply Corollary \ref{cor:translate-back-to-hypergraphs} to see that
			\[
			\frac{\widehat{Y}}{\E Y} \,\,\, \overset{d}{\longrightarrow}  \,\,\,e^{-\lambda_1 \zeta_1} 
			\, \prod_{j=2}^\infty (1+\zeta_j)^{Z_j} e^{-\lambda_j \zeta_j}.
			\]
			Applying Lemma \ref{lemma:translate}(b) shows that $\widehat{Y}$ and $Y_\cG$ have the same asymptotic distribution. Finally, combining \eqref{eq:EY_G} and Lemma \ref{lem:EY-hypergraph} gives us $\E Y_{\cG} \sim \exp(-\lambda_1 \zeta_1) \E Y$, and the result follows.
	\end{proof}
	
	%%%%%%%%%%%%%%%%%%%%%%%%%%%%%%%%%%%%%%%%%%

\appendix

\section{Technical proofs}\label{s:deferred}

\subsection{Proof of Lemma~\ref{lem:s2}}\label{sec:lem-s2}
Fix $(P_1,\ldots, P_\ell)\in\mathcal{S}^\ast(\xvec)$ and $\Ivec=(I_1,\ldots, I_\ell)
\in\cI(\xvec)$.  Let $Q = P_1\cup \cdots \cup P_\ell$.
In Step~2 we work in the irregular configuration model determined
by $\Ivec$.  Recall that $u=u(\Ivec)$
is defined by $u=u_1+\ldots +u_\ell$, where $u_i$ is the number of entries of $I_i$
which equal zero. In the irregular configuration model:
\begin{itemize}
\item There are $n-(s-1)\abs{Q}$ regular cells with $r$ points each,
\item There are $(s-2)u$ internal irregular cells with $r-1$ points each,
\item There are $u$ external irregular cells: if an external irregular cell
was collapsed from a path in $P_T\cap Q$ with $k$ edges then it
contains $(rs-r-s)k + r-2$ points, with $k\geq 0$.
\end{itemize}
This gives $n' = n-(s-1)\abs{Q}+(s-1)u$ cells in total.
The number of ways to complete Step~2 equals the number of
ways of choosing a subpartition $P'$ in this irregular configuration model
such that $T'=G(P')$ is a spanning tree.
Then $T'$ corresponds exactly to a spanning tree $T$ in the standard configuration model, with the subgraphs determined by $\Ivec$ contracted to single vertices.

We perform this count by conditioning on the degree of each vertex of $T'$. 
Label the $u$ external irregular cells in increasing order, and label the $(s-2)u$ internal irregular 
cells in increasing order.
Let $d_i$ be the number of points in the $i$'th external irregular cell; that is,
\[ d_i = (rs-r-s)k + r-2\] 
if the $i$'th external irregular cell corresponds to a path of length $k$ in
$P_T\cap Q$. For a degree sequence $\boldsymbol{\delta}$, let $\abs{\boldsymbol{\delta}}$ be its degree sum. 
Let $\cD_{\operatorname{irreg}}$ be the set of possible degree sequences for the irregular cells, and 
for a given $(\boldsymbol{\delta^{\text{ext}}},\boldsymbol{\delta^{\text{int}}}) \in \cD_{\text{irreg}}$,
let $\cD_{\operatorname{reg}}(\boldsymbol{\delta^{\text{ext}}},\boldsymbol{\delta^{\text{int}}})$ be of possible degree sequences for the regular cells: 
	\begin{align*}
	\cD_{\operatorname{irreg}} &= \{ (\boldsymbol{\delta^{\text{ext}}},\boldsymbol{\delta^{\text{int}}}) \in \N^{u} \times \N^{(s-2)u}: 1 \le \delta_i^{\text{ext}} \le d_i, \quad 1 \le \delta_i^{\text{int}} \le r-1\}, \\
	\cD_{\operatorname{reg}}(\boldsymbol{\delta^{\text{ext}}},\boldsymbol{\delta^{\text{int}}}) &= \bigg\{ \boldsymbol{\delta}^{\text{reg}} \in \N^{n - (s-1)\abs{Q}} : \sum_{i = 1}^{n - (s-1)\abs{Q}} \delta^{\text{reg}}_i = \frac{s(n'-1)}{s-1} - \abs{\boldsymbol{\delta^{\text{ext}}}} - \abs{\boldsymbol{\delta^{\text{int}}}}, \quad \delta_i \ge 1  \bigg\}.
	\end{align*}
Let $\mathcal{T}_{n'}(\boldsymbol{\delta^{\text{reg}}}, \boldsymbol{\delta^{\text{ext}}},\boldsymbol{\delta^{\text{int}}})$ be the set of trees on $n'$ vertices which have degree sequence consistent with $(\boldsymbol{\delta^{\text{reg}}}, \boldsymbol{\delta^{\text{ext}}},\boldsymbol{\delta^{\text{int}}})$. Then, the number of ways to complete Step~2 is
	\begin{equation}
\label{eq:s2-intermediate}
	s_2(\xvec,\Ivec) = \sum_{(\boldsymbol{\delta^{\text{ext}}},\boldsymbol{\delta^{\text{int}}}) \in \cD_{\text{irreg}}} \;\,\,
	\sum_{\boldsymbol{\delta^{\text{reg}}} \in \cD_{\operatorname{reg}}(\boldsymbol{\delta^{\text{ext}}},\boldsymbol{\delta^{\text{int}}})} \;\,\,
	\sum_{T' \in \mathcal{T}_{n'}(\boldsymbol{\delta^{\text{reg}}}, \boldsymbol{\delta^{\text{ext}}},\boldsymbol{\delta^{\text{int}}})} \;\,\,
	\sum_{P'_{T'}: G(P'_{T'}) = T'} \, 1.
	\end{equation}
	To simplify this, for a given $T'\in\mathcal{T}_{n'}(\boldsymbol{\delta^{\text{reg}}},\boldsymbol{\delta^{\text{ext}}},
\boldsymbol{\delta^{\text{int}}})$,
%(T', \boldsymbol{\delta}, \boldsymbol{\delta^{\text{ext}}},\boldsymbol{\delta^{\text{int}}})$, 
the number of subpartitions $P'_{T'}$ that project to $T'$ % with degree sequence $(\boldsymbol{\delta^{\text{reg}}}, \boldsymbol{\delta^{\text{ext}}},\boldsymbol{\delta^{\text{int}}})$
 is
	\[
	\bigg(\prod_{i=1}^{u} (d_i)_{\delta_i^{\text{ext}}} \bigg) \bigg(\prod_{i=1}^{(s-2)u} (r-1)_{\delta_i^{\text{int}}} \bigg)\bigg(\prod_{i=1}^{n-(s-1)\abs{Q}} (r)_{\delta^{\text{reg}}_i} \bigg).
	\]
	For a given $(\boldsymbol{\delta^{\text{reg}}}, \boldsymbol{\delta^{\text{ext}}},\boldsymbol{\delta^{\text{int}}})$, by (\ref{degree-sequence-hypertree}), the number of trees in $\mathcal{T}_{n'}(\boldsymbol{\delta^{\text{reg}}}, \boldsymbol{\delta^{\text{ext}}},\boldsymbol{\delta^{\text{int}}})$ is
	\[
	\frac{s-1}{\bigg(\displaystyle\prod_{i=1}^{u} (\delta_i^{\text{ext}}-1)! \bigg) \bigg(\displaystyle\prod_{i=1}^{(s-2)u} (\delta_i^{\text{int}}-1)! \bigg) \bigg(\displaystyle\prod_{i=1}^{n-(s-1)\abs{Q}} (\delta^{\text{reg}}_i-1)! \bigg)} \times \frac{(n'-2)!}{ \left( (s-1)! \right)^{\frac{n'-1}{s-1}}}.
	\]
Substituting these expressions into (\ref{eq:s2-intermediate}) shows that
	\begin{equation}
\label{s2I}
	s_2(\xvec,\Ivec) = \sum_{(\boldsymbol{\delta^{\text{ext}}},\boldsymbol{\delta^{\text{int}}}) \in \cD_{\text{irreg}}} A_{(\boldsymbol{\delta^{\text{ext}}},\boldsymbol{\delta^{\text{int}}})} \left( \prod_{i=1}^{u} \frac{(d_i)_{\delta_i^{\text{ext}}}}{(\delta_i^{\text{ext}}-1)!} \right) \left( \prod_{i=1}^{(s-2)u} \frac{(r-1)_{\delta_i^{\text{int}}}}{(\delta_i^{\text{int}}-1)!} \right) 
	\end{equation}
	where %, using Stirling's approximation,
\begin{align*}
	A_{(\boldsymbol{\delta^{\text{ext}}},\boldsymbol{\delta^{\text{int}}})} 
&= \frac{(s-1)(n'-2)!}{\left( (s-1)! \right)^{\frac{n'-1}{s-1}}} \sum_{\delta \in \cD_{\operatorname{reg}}(\boldsymbol{\delta^{\text{ext}}},\boldsymbol{\delta^{\text{int}}})} \prod_{i=1}^{n-(s-1)\abs{Q}} \frac{(r)_{\delta^{\text{reg}}_i}}{(\delta^{\text{reg}}_i - 1)!} \nonumber \\
	&= \frac{(s-1)(n'-2)!}{\left( (s-1)! \right)^{\frac{n'-1}{s-1}}} \left[ z^{s(n'-1)/(s-1) - \abs{\boldsymbol{\delta^{\text{ext}}}} - \abs{\boldsymbol{\delta^{\text{int}}}}} \right] \left( rz(1+z)^{r-1} \right)^{n-(s-1)|Q|} \nonumber \\
	&= \frac{(s-1)(n'-2)!\, r^{n-(s-1)\abs{Q}}}{\left( (s-1)! \right)^{\frac{n'-1}{s-1}}}  \binom{(r-1)(n-(s-1)\abs{Q})}{\frac{s(n'-1)}{s-1} - \abs{\boldsymbol{\delta^{\text{ext}}}} - \abs{\boldsymbol{\delta^{\text{int}}}} - (n-(s-1)\abs{Q})}. 
\end{align*}
(The square bracket in the second line denotes coefficient extraction.) Using Stirling's formula, we have
\begin{align}
A_{(\boldsymbol{\delta^{\text{ext}}},\boldsymbol{\delta^{\text{int}}})}
	&\sim 
A_u\, (s-1)^{(s-1)u}\, (rs-r-s)^{u-\abs{\boldsymbol{\delta^{\text{ext}}}} - \abs{\boldsymbol{\delta^{\text{int}}}}} \nonumber\\
 &= A_u\, \prod_{i=1}^u \frac{s-1}{(rs-r-s)^{\delta^{\text{ext}}_i-1}}\,\, \prod_{i=1}^{(s-2)u}
  \frac{s-1}{(rs-r-s)^{\delta^{\text{int}}_i}}
\label{AddAu}
\end{align}
where
\begin{align}
A_u  &=
 \frac{(r-1)^{1/2} \, (s-1)^2\, ((s-1)!)^{\frac{1}{s-1}}\,}
     {(rs-r-s)^{\frac{3s-1}{2(s-1)}}\,  \, n^{2}} \,
  \left(\frac{(rs-r-s)^{s-1}\, n^{s-1}}{(s-1)^{s-1}\, (s-1)!}\right)^u\nonumber \\
& \quad  \times
  \left(\frac{(rs-r-s)^{rs-r-s} \, (s-1)!}{((r-1)(s-1))^{(r-1)(s-1)} \, r^{s-1}\, n^{s-1}} \right)^{\abs{Q}}
	%\\ & \hspace{5cm} {}  \times 
  \, \left(\frac{r(r-1)^{r-1} (s-1)^{r-1}\, n}{e((s-1)!)^{\frac{1}{s-1}}\, (rs-r-s)^{\frac{rs-r-s}{s-1}}} \right)^n. 
\label{Au}
\end{align}

Next, note that
\[
\sum_{i=1}^\ell \sum_{k=0}^{m-1} \qk{I_i} = \sum_{i=1}^\ell u_i = u.
\]
By counting the cells in $Q$, we have
	\[
	(s-1)\abs{Q} = (s-2)u + \sum_{i=1}^\ell \sum_{k=0}^{m-1} ((s-1)k + 1) \,\qk{I_i}.
	\]
	It follows that 
	\[
	\sum_{i=1}^\ell \sum_{k=0}^{m-1} k \, \qk{I_i} = \abs{Q} - u.
	\] 
Therefore, 
\begin{align}
&\phantom{=} \sum_{\boldsymbol{\delta^{\text{ext}}}} \left( \prod_{i=1}^{u} \,
  \frac{(d_i)_{\delta_i^{\text{ext}}} \,(s-1)}{(\delta_i^{\text{ext}}-1)! \, (rs-r-s)^{\delta_i^{\text{ext}} - 1}} \right) \nonumber \\
	&= \prod_{k=0}^{m-1} \prod_{i=1}^\ell \left( \sum_{j = 1}^{(rs-r-s)k + r-2} \frac{\big((rs-r-s)k + r-2\big)_{j} \, (s-1)}{(j-1)! \, (rs-r-s)^{j - 1}}  \right)^{\qk{I_i}} \nonumber \\
	&= \prod_{k=0}^{m-1} \prod_{i=1}^\ell \left( \big((rs-r-s)k + r-2\big) (s-1) \,\sum_{j = 1}^{(rs-r-s)k + r-2}  \frac{\binom{(rs-r-s)k + r-3}{j-1}}{(rs-r-s)^{j - 1}}  \right)^{\qk{I_i}} \nonumber \\
	&= \prod_{k=0}^{m-1} \prod_{i=1}^\ell \left( \left(k + \frac{r-2}{rs-r-s}\right) (rs-r-s)(s-1) \left(1 + \frac{1}{rs-r-s} \right)^{(rs-r-s)k + r-3}  \right)^{\qk{I_i}} \nonumber \\
	&=  \left( \frac{(rs-r-s)^2}{r-1}\right)^u \left(\frac{(r-1)(s-1)}{rs-r-s} \right)^{(rs-r-s)\abs{Q} - (r-1)(s-2)u} \; \prod_{k=0}^{m-1} \left(k + \frac{r-2}{rs-r-s}\right)^{\sum_{i=1}^\ell \qk{I_i}}.
\label{eq:sum1}
	\end{align}
	Similarly, we have
	\begin{align}
	\sum_{\boldsymbol{\delta^{\text{int}}}} \left( \prod_{i=1}^{(s-2)u} \frac{(r-1)_{\delta_i^{\text{int}}}(s-1)}{(\delta_i^{\text{int}}-1)! (rs-r-s)^{\delta_i^{\text{int}}}} \right) 
&= \left(\sum_{j=1}^{r-1} \frac{(r-1)_{j} (s-1)}{(j-1)! (rs-r-s)^{j}} \right)^{(s-2)u} \nonumber \\
	&= \left( \frac{(r-1)(s-1)}{rs-r-s} \sum_{j=1}^{r-1} \binom{r-2}{j-1} \frac{1}{(rs-r-s)^{j -1}} \right)^{(s-2)u} \nonumber \\
	&= \left(\frac{(r-1)(s-1)}{rs-r-s} \right)^{(r-1)(s-2)u}.
\label{eq:sum2}
	\end{align}
The proof is completed by substituting (\ref{Au}), (\ref{eq:sum1})
and (\ref{eq:sum2}) into (\ref{s2I}), using (\ref{AddAu}).

\subsection{Proof of Lemma~\ref{lem:global-max}}\label{s:global-max}

	We assume that $r,s\geq 2$ are fixed integers such that $s\geq 5$ and $r > \rho(s)$,
 or $s\in \{2,3,4\}$ and $(r,s)\neq (2,2)$.
	Recall the definition of $K$ and $\varphi$ from (\ref{eq:K-hypergraphs}), (\ref{phi-def}).
	The partial derivatives of $\varphi: K \rightarrow \mathbb{R}$ are 
	\begin{align*}
	\varphi_{\alpha}(\alpha,\beta) &= \log \left(\frac{(\alpha+\beta)(r-1)(1-(s-1)\alpha-\beta)}
	{\alpha(r-1-\alpha-\beta)}\right),\\
	\varphi_{\beta}(\alpha,\beta) &= 
	\log\left( \frac{(\alpha+\beta)(r-1)}{r-1-\alpha-\beta}\right)
	+ \frac{1}{s-1} \log\left( \frac{(1-(s-1)\alpha-\beta)(rs-r-s-s\beta)}{\beta^2}\right).                  
	\end{align*}
	
	For $x\geq -1$,  let
	\begin{align*}
	\alpha(x) = \frac{1+x}{rs - r +sx + \frac{x(x+1)}{r-1}},\qquad 
	\beta(x)  =  \frac{rs-r-s}{rs -r +sx + \frac{x(x+1)}{r-1}}.
	\end{align*}
	Note that 
	\[
	\text{$(\alpha(x), \beta(x))$ lies in the interior of the domain $K$,  for any $x\in (-1,\infty)$.}
	\]
	Indeed, we have  $\alpha(x)>0$, $\beta(x)>0$, and
	\begin{equation}\label{1-sa-b}
	1 - (s-1)\alpha(x) - \beta(x) =  \frac{1 +x+ \frac{x(x+1)}{r-1}}{rs - r +sx + \frac{x(x+1)}{r-1}}
	\end{equation}
	which is strictly positive for all $x > -1$. 
	
	Our interest in this particular curve $(\alpha(x), \beta(x))$ is clarified by the following  lemma, which shows that it is a parameterisation of a ridge containing any stationary point of $\varphi$ in $K$.
	
	\begin{lemma}\label{lem1}
		Let $K$, $\varphi:K \rightarrow \mathbb{R}$, $\alpha(x)$, $\beta(x)$ be defined as above. Then, the following holds.
		\begin{itemize}
			\item[(a)]  Any local maxima of $\varphi$ on $K$ either equals $(0,0)$ or lies in the interior of $K$.
			\item[(b)] 
			For any stationary point  $(\alpha,\beta)$ in the interior of $K$ of the function $\varphi$ 
			there exists some  $x \in  (-1,\infty)$  such that $\alpha=\alpha(x)$ and $\beta = \beta(x)$ and
			\begin{equation}\label{b-eq}
			(rs-r-s) \left(1 + \dfrac{x}{r-1}\right)^{s-2} = (1+x) \left(rs-r-s + sx + \dfrac{x(x+1)}{r-1}\right).
			\end{equation}
			\item[(c)] Let $f: (-1,\infty)\rightarrow \mathbb{R}$ be defined by $f(x) = \varphi(\alpha(x), \beta(x))$. 
			If  $x \in  (-1,\infty)$  is  a stationary point of $f$ then $x$ solves \eqref{b-eq}.
		\end{itemize}
	\end{lemma}
	\begin{proof}
		To prove (a), we need to show that  none of the following points is a local maximum:
		\begin{itemize}
			\item[(i)]  $(\alpha,0)$ for all $0 <\alpha\leq \frac{1}{s-1}$;
			\item[(ii)] $(0,\beta)$ for all $0<\beta \leq 1$;
			\item[(iii)] $(\alpha,\beta)$ for all  positive $\alpha, \beta$ that $\beta =1 - (s-1)\alpha$.
		\end{itemize}
		
		For (i), observe that if $\alpha \rightarrow \frac{1}{s-1}$ then $\varphi_\alpha(\alpha,0) = O(1) + \log(1-(s-1)\alpha) \rightarrow -\infty$. 
		Hence $ \varphi(\alpha,0) > \varphi(\frac{1}{s-1},0)$ for sufficiently large 
		$\alpha < \frac{1}{s-1}$.
		Next, if $0<\alpha<\frac{1}{s-1}$ and $\beta\rightarrow 0$ then $\varphi_\beta(\alpha,\beta) = O(1) - \frac{2}{s-1}\log(\beta) \rightarrow +\infty$. Therefore, $\varphi(\alpha,\beta)>\varphi(\alpha,0)$ for sufficiently small positive $\beta$. 
		
		For (ii),   observe that if  $\beta \rightarrow 1$ then $\varphi_{\beta}(0,\beta) = O(1) + \frac{1}{s-1}\log(1-\beta)\rightarrow -\infty$.  Hence
		$\varphi(\beta,0)>\varphi(1,0)$ for sufficiently large $\beta<1$. Next, if $0<\beta<1$ and 
		$\alpha \rightarrow 0$ then $\varphi_{\alpha}(\alpha,\beta)  = O(1) - \log(\alpha) \rightarrow -\infty$. 
		Therefore, $\varphi(\alpha,\beta)>\varphi(\alpha,0)$ for sufficiently small positive $\alpha$. 
		
		Finally, for (iii), observe that for fixed $\alpha$ with  $0 < \alpha < \frac{1}{s-1}$ and  $\beta \rightarrow 1-(s-1)\alpha$
		we have 
		$\varphi_{\alpha}(\alpha,\beta) = O(1) + \log(1-(s-1)\alpha-\beta) \rightarrow -\infty$. 
		So $\varphi(\alpha,1-(s-1)\alpha-\varepsilon) > \varphi(\alpha,1-(s-1)\alpha)$ 
		for sufficiently small positive $\varepsilon$. This completes the proof of (a).
		
		Next we proceed to (b). Let $(\alpha,\beta)$ be a stationary point of $\varphi$ in the interior of $K$. 
		We put
		\begin{equation}\label{def_t}
		x =  \frac{\alpha(rs-r-s)}{\beta} -1.
		\end{equation}
		Clearly $x \in (-1,\infty)$ since both $\alpha$ and $\beta$ are positive. From $\varphi_\alpha(\alpha,\beta)=0$, we find that
		\[
		1-(s-1)\alpha - \beta = \dfrac{\alpha}{\alpha + \beta} - \dfrac{\alpha}{r-1}.
		\]
		Substituting $\beta = \frac{\alpha(rs-r-s)}{x+1}$, we find that
		\[
		1 - \alpha \left(s-1 + \dfrac{rs-r-s}{x+1}\right) =  \dfrac{x+1}{ x+1 + rs-r-s} - \dfrac{\alpha}{r-1}.
		\]
		After rearranging, we see that this identity is equivalent to  $\alpha = \alpha(x)$. Then, from \eqref{def_t} we find that $\beta = \beta(x)$.
		Hence
			\[ \alpha + \beta = \frac{(rs-r-s+x+1)\, \alpha}{x+1}\]
			and
		\[
		\frac{r-1-\alpha-\beta} {(r-1)(\alpha+\beta)}  = 1 + \frac{x}{r-1}.
		\]
		The definition of $\beta(x)$ implies that
		\[
		\frac{rs-r-s-s\beta}{\beta}
		= rs-r-s+sx + \frac{x(x+1)}{r-1}
		\]
		while \eqref{1-sa-b} implies that
		\[
		\frac{1-(s-1)\alpha-\beta}{\beta} =  
		\frac{(1+x)\left( 1+ \frac{x}{r-1}\right) }
		{ rs-r-s}.
		\]
		Substituting the above expressions into $\varphi_\beta(\alpha,\beta)=0$ leads to equation \eqref{b-eq} and, thus, completes the proof of (b).  
		
		Next, observe that after much rearranging, $\varphi_\alpha(\alpha(x),\beta(x))=0$ for any $x\in(-1,\infty)$.  Therefore
		\[
		f'(x) = \varphi_\beta(\alpha(x),\beta(x))\, \beta'(x).
		\] 
But $\beta(x)$ is strictly decreasing, so $f'(x) =0$ if and only if $\varphi_\beta(\alpha(x),\beta(x))=0$. 
			Therefore~(c) follows from~(b). 
	\end{proof}
	
	Next, we show  that equation \eqref{b-eq} has  at most two solutions.  
	\begin{lemma}\label{lem2} 
Let $r,s\geq 2$ be fixed integers.
   Equation \eqref{b-eq} has a solution at $x=0$. Moreover, the following holds.
		\begin{itemize}
			\item[\emph{(a)}]
If $s \in \{2,3,4\}$ and  $(r,s)\neq (2,2)$ then $x=0$ is the unique solution
			of equation \eqref{b-eq} on $(-1,\infty)$.  
			\item[\emph{(b)}]
			If $s\geq 5$ and $r\geq s-1$ then equation
			\eqref{b-eq} has no solutions on $(-1,0)$ and at most one solution on $(0,\infty)$.
		\end{itemize}
	\end{lemma}
	\begin{proof}
		Let functions $L, R: (-1,\infty)\rightarrow \mathbb{R}$ stand for the LHS and the RHS of \eqref{b-eq}:
		\[
		L(x) = (rs-r-s) \left(1 + \dfrac{x}{r-1}\right)^{s-2}, \qquad  
		R(x)= (1+x) \left(rs-r-s + sx + \dfrac{x(x+1)}{r-1}\right).
		\]
		Observe that $L(0) = R(0) = rs-r-s$, so $x=0$ is a solution of  \eqref{b-eq}.

		First, assume that $s\in \{2,3\}$.  Then the function $(1+x)^{-1}L(x)$ decreases on $(-1,\infty)$. Note also that 
		$(1+x)^{-1} R(x)$ is a strictly increasing function on $(-1,\infty)$. Thus, there are no other solutions 
		of \eqref{b-eq} except $x=0$. 	 	
		Similarly, for $s=4$, the function $(1+x)^{-1} \left(1 + \frac{x}{r-1}  \right)^{-1}L(x)$ decreases on $(-1,\infty)$. On the other hand,
  	\[
	\frac{R(x)}{(1+x) \left(1 + \frac{x}{r-1}  \right)} = 
		rs-r-s+2+ x - \frac{2}{1 + \frac{x}{r-1}}
		\]
		is strictly increasing on $(-1,\infty)$. Part (a) follows.
		
		We proceed to the case $s\geq 5$ and $r\geq s-1$. As  above, $(1+x)^{-1} R(x)$ is a strictly increasing function on $(-1,\infty)$.	 The function $(1+x)^{-1}L(x)$ decreases on $(-1,0)$ because
		\[
		\left((s-2) \log  \left(1 + \dfrac{x}{r-1}\right) - \log(1+x)\right)' =
		\frac{s-2}{r-1+x} - \frac{1}{1+x} < \frac{s-r-1}{(r-1+x)(1+x)}\leq 0. 
		\]
		This proves that \eqref{b-eq} has no solutions on $(-1,0)$.	 
		
		Next, we compute 
		\begin{align*}
		L'(0) &= \frac{(rs-r-s)(s-2)}{r-1}<rs-r + \frac{1}{r-1}= R'(0);\\
		L''(x) &= \frac{(rs-r-s) (s-2)(s-3)}{(r-1)^2} \left(1 + \dfrac{x}{r-1}\right)^{s-4};\\
		R''(x) &= 2s + \frac{4 + 6x}{r-1}.
		\end{align*}
		First suppose that $L''(x) < R''(x)$ for all $x\in [0,\infty)$. Then $L'(x)-R'(x)<L'(0)-R'(0)<0$. This implies that the function $L(x)-R(x) < L(0)-R(0)=0$ so \eqref{b-eq} has no solutions on $(0,\infty)$. 
		For future reference, note that this case holds when $s=5$ and $r\geq 5$, as can be verified directly.
		
		Otherwise, let $x^* = \inf \{x \in [0,\infty) \, : \,   L''(x) \geq R''(x)\}$. As before,   $x=0$ is the unique solution $L(x)=R(x)$ on $[0,x^*]$ because the function $L'(x)-R'(x)$ is strictly decreasing on this interval. In particular, we get 
		\[	
		L(x^*) - R(x^*) \leq x^* (L'(0)-R'(0)) \leq 0.
		\] 
		By continuity, we have 
		$x^* \in [0,\infty)$ and  $L''(x^*)\geq R''(x^*)$.	 
Observe that
		\[
		L^{(3)}(x^*) = L''(x^*) \frac{s-4}{r-1 + x^*} \geq R''(x^*) \frac{3}{s(r-1) + (2+3x^*)}  = R^{(3)}(x^*).
		\] 
		Note that $ R^{(3)}(x) $ is a constant.  For $s\geq 6$ the function $L^{(3)}(x)$ is strictly increasing.
When $s=5$ we need only consider $r=4$, since $r\geq 5$ is covered by the earlier argument.
			Here $L^{(3}(x)$ is also a constant and we can check directly that $L^{(3)}(x) > R^{(3)}(x)$.
		In all cases, we conclude that   $L^{(3)}(x) -  R^{(3)}(x) > 0$ for any $x\geq x^*$. Therefore,  $L''(x) -  R''(x) > 0$ for any $x>x^*$ so the function $L(x)-R(x)$ is strictly convex on $(x^*,\infty)$.  Since $L(x^*)-R(x^*)\leq 0$ we conclude that $L(x)- R(x) =0$ for at most at one point  $x\in (x^*,\infty)$. This completes the proof of (b).
	\end{proof}

\medskip
	
We show that $(\alpha_0,\beta_0)$ is a local maximiser of $\varphi$ and that the Hessisan at this 
point is strictly negative definite.
\begin{lemma}
\label{lem:negative-definite}
	Fix integers $r,s\geq 2$ such that \emph{(\ref{r-lower})} holds.
Then $(\alpha_0,\beta_0)$ is a local maximiser of $\varphi$ on $K$ and the Hessian evaluated
at the point $(\alpha_0,\beta_0)$ is strictly negative definite.
\end{lemma}
\begin{proof}
Direct substitution shows that $(\alpha_0,\beta_0)$ is a stationary point of $\varphi$. 
We will show that the Hessian $H_0$ at the point $(\alpha_0,\beta_0)$ has a 
positive determinant and a negative trace, and is therefore strictly negative definite. This will also imply that $(\alpha_0,\beta_0)$ is 
a local maximiser.
Now
			\begin{align}
			\det(H_0) &= \frac{r^3 (s-1)^2 \left(r^2-r s+r+s-1\right)}{(r-1)^2 (r s-r-s)}, 
\label{H0-det}\\
			\tr(H_0) &= -\left(\frac{r^2}{(r-1) (r s-r-s)^2}+\frac{r(2 r-1)}{(r-1) (r s-r-s)}+\frac{\left(r^2-4 r+1\right) r}{(r-1)^2}+r s(s-1)\right).\nonumber
			\end{align}
Recalling (\ref{eq:det-inequality}) we see that $\det(H_0)>0$.
To show $\tr(H_0)<0$, observe that every term inside the parentheses is positive when 
$r\geq 4$.  The only other cases are $(r,s)\in \{ (2,3),\, (3,2),\, (3,3)\}$,
and direct substitution shows that $\tr(H_0)$ is also negative in these cases.
	\end{proof}

We can now prove Lemma~\ref{lem:global-max}.

	\begin{proof}[Proof of Lemma~\ref{lem:global-max}]
The assumptions of the lemma imply that (\ref{r-lower}) holds.
	First, observe that $\alpha_0 = \alpha(0)$ and $\beta_0=\beta(0)$. 
Next, observe that the condition $\varphi(0,0) < \varphi(\alpha_0,\beta_0)$ holds if and only if 
  $L(r,s)>0$, and $L(r,s)>0$ by the assumptions on $(r,s)$, using Lemma~\ref{small-values} and 
Lemma~\ref{lem:unique-threshold}.
		
	Let  $(\alpha,\beta)$ be any global maximum of $\varphi$ on $K$. 
	By Lemma~\ref{lem1}(a),(b) and the assumption that $\varphi(0,0)<\varphi(\alpha_0,\beta_0)$, we conclude 
	that $(\alpha,\beta)$ lies in the interior of $K$ and  
	$\alpha = \alpha(x)$, $\beta=\beta(x)$ for some $x\in (-1,\infty)$. Since  $(\alpha,\beta)$  is a global 
		maximum of $\varphi$, it follows that $x$ is a global maximum of $f(x)= \varphi(\alpha(x),\beta(x))$. 
		Similarly, by assumption, $0$ is a local maximum of $f(x)$. If $x \neq 0$ then
		the function $f$ would have another stationary point between $0$ and $x$, but this is impossible by Lemma \ref{lem2}. Thus $x=0$, which shows that $(\alpha_0,\beta_0)$ is the
unique maximum of $\varphi$ on $K$.

Now $\det(-H_0) = (-1)^2\det(H_0)$ is given in (\ref{H0-det}), matching the value given in the
statement of Lemma~\ref{lem:global-max}, while direct substitution shows that that the value of 
$\varphi(\alpha_0,\beta_0)$ stated in Lemma~\ref{lem:global-max} is correct. This completes the proof. %of Lemma~\ref{lem:global-max}.}
	\end{proof}
	
\subsection{Proof of Lemma~\ref{lem:unique-threshold}}\label{s:unique-threshold}

Recall the function $L$ defined at the start of Section~\ref{s:threshold}.
		For reference, the first and second derivatives of $L_s(r)$ with respect to $r$ are
		\begin{align}
		L_s'(r) &= \frac{1}{r} + \log(r-1) - \frac{s-1}{s} \log r - \frac{1}{s} \log\left(r - \frac{s}{s-1}\right), \label{deriv1}\\
		L_s''(r) &= \frac{1}{r^2} \left( \frac{1}{r-1} - \frac{r}{rs-r-s} \right). \label{deriv2}
		\end{align}
		When $s=5$ we have $L_5(2) < 0$ and $L_5'(r) > 0$ for $r\geq 2$. 
Hence Lemma~\ref{lem:unique-threshold} holds when $s=5$.
		
		Now suppose that $s\geq 6$.  It follows from (\ref{deriv2}) that
		$L_s''(r)\geq 0$ when 
\[ r\leq r_0 = \nfrac{1}{2}\big(s+\sqrt{s(s-4)}\big),\] and $L_s''(r) < 0$
		otherwise.     The point of inflection $r_0$ satisfies $2 < s/2 < r_0 < s$.
		Next, using (\ref{in:trick}) with $a=-1$, $b=s$, we have
		\[
		e \left(1 - \frac{1}{s} \right)^s > \frac{s-1}{s} > \frac{s-2}{s-1}.
		\]
		This implies that $L_s'(s) > 0$ when $s\geq 6$.    It follows that the maximum of $L_s(r)$
		on $[2,s]$ is either $L_s(2)$ or $L_s(s)$.
		
		Next, observe that if $(s-1)^s < 2^{(s-2)^2}$ then $L_s(2) < 0$.
		This sufficient condition holds when $s=6$, and if $(k-1)^k < 2^{(k-2)^2}$ for some $k\geq 6$
		then
		\[
		k^{k+1} = k \left(\frac{k}{k-1} \right)^k (k-1)^k < 2^{k-3} \, 2^k \, 2^{(k-2)^2} = 2^{(k-1)^2}.
		\]
		Hence, by induction, $L_s(2) < 0$ for all $s\geq 6$.  Furthermore, we claim that $L_s(s) < 0$ for all
		$s\geq 6$.  Now
		\[
		L_s(s) = \log\left( \frac{(s-1)^{s}}{s^{\frac{s(s-2)}{s-1}} (s-2)^{\frac{s-2}{s-1}}} \right)
		\]
		and direct substitution shows that $L_6(6) < 0$.  When $s\geq 7$, note that
		\[
		\frac{(s-1)^{s}}{s^{\frac{s(s-2)}{s-1}} (s-2)^{\frac{s-2}{s-1}}} < \dfrac{s}{e(s-2)}\,  
		\big( s(s-2) \big)^{\frac{1}{s-1}},
		\]
		and observe that the right hand side is a decreasing function of $s$.
		Therefore for $s\geq 7$,
		\[
		L_s(s) < \log\left( \frac{7}{5e}\, 35^{1/6}\right) < 0.
		\]
		This establishes that $L_s(r) < 0$ for all $r\in [2,s]$.
		
		Applying (\ref{in:trick}) with $a = -1$ we see that when $r\geq s$,
		\[
		e^{s}\left(1-\frac{1}{r} \right)^{rs} \ge \left(1 - \frac{1}{r} \right)^s \ge \left(1 - \frac{1}{r} \right)^r > \left(1 - \frac{s}{r(s-1)} \right)^r.
		\]
		This inequality is equivalent to $L_s'(r) > 0$, so $L_s(r)$ is strictly monotonically increasing on
		$r\geq s$.  Finally,
		\[ \frac{rs-r-s}{s}\log\left(\frac{r-1}{r}\right) \to -\frac{s-1}{s},\qquad
		\frac{rs-r-s}{s(s-1)}\log\left(\frac{r-1}{r-s/(s-1)}\right) \to \frac{1}{s-1},\]
		so
		\[ \lim_{r\to\infty} L_s(r)= 
		\frac{1}{s-1}\log(s-1) - \frac{s-2}{s-1} +  \lim_{r\to\infty} \frac{1}{s-1}\log(r-1)  = \infty.\]

\subsection{Proof of Lemma~\ref{lem:bounds}}\label{s:lem-bounds}

The proof of Lemma \ref{lem:unique-threshold} showed that $L_s(r)$ is monotonically increasing for 
		$r \ge s\geq 5$.  Hence, for the first statement it suffices to show that 
		\begin{equation} \label{eq:sufficient}
		L_s(\rho^{-}(s)) <0 \qquad \text{and} \qquad L_s(\rho^{+}(s)) >0.
		\end{equation}
		These inequalities hold for $s\in \{5,6,\ldots, 11\}$, as may be verified by
direct computation (see Table~\ref{verified-directly}).
		\begin{table}[ht!]
				\renewcommand{\arraystretch}{1.5}
				\small
				\centering
				\begin{tabular}{|c||c|c|c|c|c|c|c|}
					\hline
					$s$ & $5$ & $6$ & $7$ & $8$ & $9$ & $10$ & $11$ \\ \hline
					$L_s\big(\rho^{-}(s)\big)$ & $-0.0051$ & $-0.0027$ & $-0.0012$ & $-0.00047$ & $-0.00018$ & $-0.000066$ & $-0.000025$ \\ \hline
					$L_s\big(\rho^{+}(s)\big)$ & $0.012$ &$0.0039$ & $0.0013$ & $0.00045$ & $0.00016$ & $0.000057$ & $0.000021$  \\
					\hline
				\end{tabular}
				\caption{The values of $L_s(\rho^{-}(s))$ and $L_s(\rho^{+}(s))$
for $s=5,\dots,11$, to $2$ significant figures}
\label{verified-directly}
			\end{table}
		For the remainder of the proof of the first statement, assume that $s\geq 12$.
		
	To prove the upper bound, let $f(x) = \log(1-x)$ and $R_2(x) = f(x) + x + \frac{x^2}{2}$.
	We write
		\begin{align*}
		& L_s\big(\rho^{-}(s) \big) \\
		&= \frac{s-2}{s-1} 
		+ \left(\frac{e^{s-2}}{s-1} - \frac{s}{2} - \frac{1}{2} \right) \log\left(1 - \frac{(s+1)(s-1)}{2e^{s-2}} \right) \\
		&\qquad {} -  \left( \frac{e^{s-2}}{s} - \frac{s}{2} - \frac{1}{2s} \right) \log\left(1 - \frac{(s-1)^2}{2e^{s-2}} \right)  - \left(\frac{e^{s-2}}{s(s-1)}- \frac{s^2+1}{2s(s-1)} \right) \log\left(1 - \frac{s^2+1}{2e^{s-2}} \right) \\
		&= \frac{s-2}{s-1} \\ & {} 
		+ \left(\frac{e^{s-2}}{s-1} - \frac{s}{2} - \frac{1}{2} \right) \left(- \frac{(s+1)(s-1)}{2e^{s-2}} - \frac{1}{2} \left(\frac{(s+1)(s-1)}{2e^{s-2}}\right)^2 + R_2\left(\frac{(s+1)(s-1)}{2e^{s-2}} \right) \right) \\
		& \quad -  \left( \frac{e^{s-2}}{s} - \frac{s}{2} - \frac{1}{2s} \right) \left( - \frac{(s-1)^2}{2e^{s-2}} - \frac{1}{2} \left(\frac{(s-1)^2}{2e^{s-2}} \right)^2 + R_2\left(\frac{(s-1)^2}{2e^{s-2}} \right) \right) \\ 
		& \quad - \left(\frac{e^{s-2}}{s(s-1)}- \frac{s^2+1}{2s(s-1)} \right) \left(- \frac{s^2+1}{2e^{s-2}} - \frac{1}{2} \left( \frac{s^2+1}{2e^{s-2}} \right)^2 + R_2\left(\frac{s^2+1}{2e^{s-2}} \right)\right)
		\end{align*}
		All inputs to $R_2(\cdot)$ in the above expression lie in $[0,\frac{145}{2e^{10}}]$ as $s\geq 12$, and
		$-c_1 \le f'''(x) \le -2$ where $c_1 = \frac{16 e^{30}}{\left(2 e^{10}-145\right)^3} \approx 2.01988$. 
		Hence, by Taylor's Theorem, we conclude that
		$-c_1 x^3/6 \le R_2(x) \le - x^3/3$.
		This implies that
		\begin{align*}
		&  96(s-1)e^{3(s-2)}\, L_s\big(\rho^{-}(s)\big)  \\
		&\qquad \le -  \Big(2e^{s-2} \big(24 s e^{s-2} - c_1(s^2-1)^3\big)
		+ (s^2-1)^3 \big((c_1-2)s^2 + c_1+2\big) \\
		& \hspace*{6cm} {} +4e^{s-2} \big(s^6-6 s^5+18 s^4-18 s^3+21 s^2-12 s+8\big) 
		\Big),
		\end{align*}
		which is negative %as
		when $s\geq 12$.  This establishes the first inequality in (\ref{eq:sufficient}).
		
		The second inequality follows similarly, writing
		\begin{align*}
		& L_s\big( \rho^{+}(s) \big) \\
		&= \frac{s-2}{s-1} + \left(\frac{e^{s-2}}{s-1}-\frac{s-1}{2} \right) \left(- \frac{(s-1)^2}{2e^{s-2}} - \frac{1}{2} \left( \frac{(s-1)^2}{2e^{s-2}}\right)^2+ R_2\left(\frac{(s-1)^2}{2e^{s-2}} \right) \right) \\
		&{} - \left( \frac{e^{s-2}}{s} - \frac{s}{2} + 1 - \frac{3}{2s} \right) \left(- \frac{(s-1)(s-3)}{2e^{s-2}} - \frac{1}{2} \left(\frac{(s-1)(s-3)}{2e^{s-2}} \right)^2 + R_2\left(\frac{(s-1)(s-3)}{2e^{s-2}} \right) \right) \\
		&{} - \left(\frac{e^{s-2}}{s(s-1)}- \frac{s^2-2s+3}{2s(s-1)} \right) \left(- \frac{s^2-2s+3}{2e^{s-2}} - \frac{1}{2} \left( \frac{s^2-2s+3}{2e^{s-2}}\right)^2 + R_2\left(\frac{s^2-2s+3}{2e^{s-2}} \right) \right) 
		\end{align*}
		and using the bound $-c_2 x^3/6 \le R_2(x) \le - x^3/3$
		for $x \in [0,\frac{123}{2e^{10}}]$, where
		$c_2 = \frac{16 e^{30}}{\left(2 e^{10}-123\right)^3} \approx 2.01685$. 
		This leads to
		\begin{align*}
		& 96(s-1)e^{3(s-2)}\,  L_s\big( \rho^{+}(s) \big) \\
		&\quad \ge 2e^{s-2} \big(24 e^{s-2} (s-2) - c_2(s-1)^6 \big) 
		+4e^{s-2} \big(s^6-6 s^5+18 s^4-34 s^3+29 s^2-12\big) 
		\\ &  \qquad {} 
		+ c_2 (s-1)^8 -2 \big(s^8-14 s^7+90 s^6-334 s^5+796 s^4-1258 s^3+1302 s^2-810 s+243\big)
		\end{align*}
		which is positive %, as
		when $s \ge 12$. 
This concludes the proof of the first statement of the lemma.  

For the
second statement, observe from (\ref{rho-def}) that the definition of $\rho(s)$ 
can be rewritten as
\begin{equation}\label{eq1}
	\frac{\log(s-1)+\log(\rho-1)}{\rho s - \rho - s} =  
	  \frac{s-1}{s} \log\left(1+ \frac{s}{\rho s - \rho - s}\right) - 
	  \log\left(1+ \frac{1}{\rho s - \rho - s}\right).
\end{equation}
It follows from the first statement of Lemma~\ref{lem:bounds} that
%We also know that 
%\[
	 %\frac{e^{s-2}}{s-1} - \frac{s-1}{2}< \rho < \frac{e^{s-2}}{s-1} - \frac{s-3}{2}.
%\]
\[ \rho s - \rho - s = e^{s-2}\big(1 + O(s^3\, e^{-s})\big).\]
Using Taylor's theorem, we find that, as $s\rightarrow \infty$,
\begin{align*}
 (\rho s - \rho - s) \log\left(1+ \frac{s}{\rho s - \rho - s}\right) &= 
 s - \frac{s^2}{2(\rho s -\rho - s)} + O(s^3 e^{-2s})
 \\
 &=   s  -  \frac{s^2 }{ 2e^{s-2} } + O(s^5 e^{-2s}).
 \end{align*}
 Similarly, we have
 \[
    (\rho s - \rho - s)  \log\left(1+ \frac{1}{\rho s - \rho - s}\right) =
   1 -   \frac{1} {2e^{s-2}} + O(s^3e^{-2s}).
\]
Substituting these bounds into \eqref{eq1}, we find that
\begin{align*}
(s-1)(\rho-1) &= \exp\left(s-2   - \frac{s(s-1) }{ 2  e^{s-2} } + \frac{1} {2e^{s-2}} + O(s^5 e^{-2s})  \right)
\\ &= e^{s-2} -   \frac{(s^2-s-1) }{ 2}  + O(s^5 e^{-s}).
\end{align*}
The proof is completed by solving for $\rho$.

\end{document}